\documentclass[10pt,oneside,shortlabels]{amsart}
\usepackage[T1]{fontenc}
\usepackage{amsthm,amsmath}
\usepackage{amsbsy}
\usepackage{amstext}
\usepackage{amssymb}
\usepackage{esint}
\usepackage{verbatim}
\usepackage{enumitem}

\makeatletter
\numberwithin{equation}{section}
\numberwithin{figure}{section}

\usepackage{mathtools}

\newtheorem {theo} {Theorem} [section]
\newtheorem {prop} [theo] {Proposition}

\newtheorem {defi} [theo] {Definition}

\newtheorem*{thmC}{Theorem}
\theoremstyle{definition}
\newtheorem{example}{Example}
\newtheorem {obs} [theo] {Remark}

\newcommand{\R}{\ensuremath{\mathbb{R}}}

\def\sideremark#1{\ifvmode\leavevmode\fi\vadjust{\vbox to0pt{\vss % the remark
    \hbox to 0pt{\hskip\hsize\hskip1em           %                will appear only
 \vbox{\hsize2cm\tiny\raggedright\pretolerance10000%                on the side
 \noindent #1\hfill}\hss}\vbox to8pt{\vfil}\vss}}}%

\usepackage[all]{xy}        %     appelle le package
\newcommand{\xyL}[1]{%
	\xydef@\xymatrixrowsep@{#1}
} 
\newcommand{\xyC}[1]{%
	\xydef@\xymatrixcolsep@{#1}
} 
\makeatother

\begin{document}

\title{The Fatou coordinate for parabolic Dulac germs}

\author{P. Marde\v si\'c$^{1}$, M. Resman$^{2}$, J.-P. Rolin$^{3}$,
V. \v Zupanovi\'c$^{4}$}
\begin{abstract}

We study the class of parabolic Dulac germs of hyperbolic polycycles. For such germs we give a constructive proof of the existence of a \emph{unique} Fatou coordinate, admitting an asymptotic expansion in the power-iterated logarithm monomials. 
\begin{comment}
Motivated by the construction of the Ecalle-Voronin moduli for germs of analytic parabolic diffeomorphisms,  we expect the result to be useful in view of the analytic classification of Dulac germs. The question is also interesting in itself, directly related to the question of embedding of a Dulac germ in a flow.
\end{comment}

\end{abstract}
\maketitle
\noindent \emph{Acknowledgement}. This research was supported by:
Croatian Science Foundation (HRZZ) project no. 2285, French ANR project
STAAVF, French-Croatian bilateral Cogito project 33003TJ \emph{Classification
de points fixes et de singularit\'es \`a l'aide d'epsilon-voisinages
d'orbites et de courbes}, Croatian UKF project \emph{Classifications of Dulac maps and epsilon-neighborhoods 2018, the University of Zagreb research support
for 2015 and 2016. }

\section{\label{sec:introduction}Introduction}
A classical
problem  for germs  of diffeomorphisms on the real or complex line  at a fixed point  is the problem of embedding in the flow of vector fields or, equivalently, the construction of a Fatou coordinate in which the diffeomorphisms is the translation by one. 

The problem is solved in most cases in the analytic setting. Denote by $f$ the analytic germ of a diffeomorphism, and by $\lambda$ the multiplier of the fixed point. The problem is fully solved in the following two cases:  in the hyperbolic case
 ($|\lambda|\ne1$) by K\"{o}nig's and B\"{o}tcher's linearization theorem; in the parabolic, i.e. tangent-to-identity, case ($\lambda=1$, to which also the case of $\lambda$ being a root of unity is reduced) by the Leau-Fatou theorem, see e.g. \cite{milnor}.

We solve the problem of existence and uniqueness of Fatou coordinates in a non-analytic case:
for parabolic Dulac germs. 
Dulac germs are germs $f$ on $(\R^+,0)$ analytic outside of zero, having the well-known \emph{Dulac} power-log asymptotic expansion $\hat{f}$ and which are moreover \emph{quasi-analytic}, i.e. which admit an extension to a complex domain ensuring the injectivity of the mapping $f\mapsto \hat{f}$  (see Definition~\ref{def:dulac_germ}). In fact, we build a Fatou coordinate which admits a \emph{transasymptotic expansion} in power-logarithm monomials (see Section \ref{subsec:tP}). This guarantees its uniqueness.

Our work on Dulac germs is motivated by 
the study of cyclicity of polycycles of analytic vector fields. 
Indeed, it is proved by Ilyashenko \cite{ilyalim} that the first-return maps of hyperbolic polycycles are Dulac germs in the sense of Definition~\ref{def:dulac_germ}.
The class of parabolic Dulac germs is the most interesting case in the study of the \emph{cyclicity}
 of  polycycles. 
 Recall that the cyclicity is the maximal number of limit cycles which can appear in a neighborhood of the polycycle in an analytic deformation. Cycles are given by zeros of the \emph{displacement germ} $\Delta(x)$, that is the Dulac germ minus the identity. In the case of a non-parabolic Dulac germ, the displacement germ is of the form $\Delta(x)=ax+o(x)$, with $a\ne0$ \cite{cherkas}. In the parabolic case the displacement function is flatter than a function having a linear term, so one should expect higher cyclicity. This is indeed a theorem in the case of \emph{hyperbolic loops} (\emph{i.e.} polycycles with one vertex): there exist parabolic Dulac germs corresponding to homoclinic loops of arbitrary high cyclicity \cite{roussarie_number}, whereas in the non-parabolic case the cyclicity is one \cite{cherkas}. In \cite{mourtada}, universal bounds for cyclicity of hyperbolic polycycles of $2$, $3$ and $4$ vertices are given under some generic conditions. They imply that the corresponding Dulac germ is non-parabolic. The cyclicity problem in the parabolic case, even for polycycles with such low number of vertices, is widely open.\\

\begin{comment}
Recall that for analytic tangent-to-the identity germs of diffeomorphisms, the construction of the Fatou coordinate was the keypoint of the construction of the well-known Ecalle-Voronin moduli of analytic classification, see 
\end{comment}

%Embedding results in non-analytic setting
%are important, in particular, in the theory of codimension
%one foliations. In that spirit, codimension
%one foliations are presented as \emph{dimension one dynamical systems} in \cite{ghys_tsuboi:conjugaisons_systemes_dynamiques}, because of the essential role played by their \emph{transverse pseudo-group}. Actually, our Dulac germs play a similar role, as they describe dynamically the \emph{transverse behavior} of a planar analytic vector field $X$ in a neighborhood of a hyperbolic polycycle of $X$.\\

The problem of embedding of diffeomorphisms in a flow has been studied previously in various non-analytic contexts.  For a tangent-to-identity
$\mathcal{C}^{\infty}$ germ $f$ with a \emph{non-flat contact} with
the identity, the embedding of $f$ in the flow of a smooth vector
field is proved by Takens \cite[Theorem~4]{takens:normal_forms}. Recall that a germ $f$ has a \emph{flat contact with the identity} if $f(x)-x$ tends to $0$ faster than any power of $x$ as $x\to0$.
The case of a parabolic $\mathcal{C}^{\infty}$ germ with a \emph{flat
	contact} with the identity and 0 as an isolated fixed point, and an
extra non-oscillation condition, is considered by Sergeraert \cite[Theorem~3.1]{serge}.
This result can be applied, for example, in the study of conjugation
classes of some $\mathcal{C}^{\infty}$ codimension one foliations
of the cylinder $\mathbb{S}^{1}\times\mathbb{R}$. Finally, if $f$
is a germ of $\mathcal{C}^{r}$ diffeomorphism of $\mathbb{R}^+$,
$r\ge2$, with $0$ as an isolated fixed point, then a classical result
of Szekeres says that $f$ embeds in the flow of a $\mathcal{C}^{1}$
vector field, which is $\mathcal{C}^{r-1}$ outside $0$ \cite{szekeres:regular_iteration}.\\

Note that none of these results applies as such in our framework.
First, a parabolic Dulac germ $f$, while having analytic representatives
on open intervals $\left(0,d\right)$, is not in general $\mathcal{C}^{\infty}$
at $0$. Moreover, thanks to the \emph{quasi-analyticity} result of
\cite{ilyalim}, if $f$ is not equal to the identity, then
its Dulac asymptotic expansion is different from $x$. Hence,
unlike the situation considered by Sergeraert, $f$ does not have a flat
contact with the identity. Finally, Szekeres' result could be applied
to Dulac germs of class $\mathcal{C}^{2}$ at $0$. However, Szekeres' methods do not lead to an asymptotic expansion of the Fatou coordinate in the power-iterated log scale. 
%This is crucial for further applications we have in mind. 

\begin{comment}
in view of the fractal
analysis of the orbits of $f$, which is our ultimate goal, it is crucial
for us to construct a Fatou coordinate \emph{which admits an asymptotic
expansion in $$w=\Psi_{0}\left(z\right)\widehat{\mathfrak{L}}$$}. Szekeres' methods would
not lead to such expansions. \\
\end{comment}

In our study, the main difficulty consists in giving a meaning to the notion of transasymptotic expansion. In the classical case of Dulac germs, the problem does not occur as, in the expansions, each power of $x$ is multiplied by a \emph{polynomial}
function in $\log(x)$. However, this is not the case anymore for the Fatou coordinate \cite{mrrz2}. In order to build a Fatou coordinate, we perform a transfinite version of the classical Poincar\'e algorithm. In each step, we solve the same Abel equation on the formal and on the germ level. To this end, each time a power of $x$ is multiplied in the expansion by an infinite series, we choose an appropriate representative of this series, consistently with the Abel equation. We call \emph{integral section} this suitable choice of a representative  (see Definitions \ref{sectional} and \ref{defi}).\\

The problem of the choice of a germ represented by an infinite series is a key problem in the study of analytic dynamical systems. The problem appears in Ilyashenko's solution of Dulac's problem. The solution  is given by imposing the existence of an analytic
%\edz{I have removed 'bounded', in fact, $f(e^{-Z})$ is bounded but when we pre-compose by log, it is again parabolic germ on SQD, so not bounded.}
extension of the germ to a sufficiently big complex domain (quasi-analyticity), \cite{ilyalim}. In our construction of the Fatou coordinate,
 the successive choices are done by imposing to the germs that appear in the process to be solutions of Abel equations. 
\smallskip 

\paragraph{\bf{Perspectives}}
Analytic classification of parabolic analytic germs was given by Ecalle-Voronin moduli (see \cite{ecalle, voronin} or \cite{loray} for an overview). It was given by comparison of Fatou coordinates on corresponding sectors in the complex plane. We want to extend Ecalle-Voronin moduli to Dulac germs. We hence need Fatou coordinates defined on sufficently large complex sectors. 
Our main result  is formulated on the real line (i.e. on $(0,d),\ d>0$). 
However, by Ilyashenko \cite{ilya}, the Dulac germ extends to a \emph{standard quadratic domain} in the complex domain. 
In this paper, we construct the Fatou coordinate in a small sector in the complex domain containing the interval $(0,d)$. Extension of the construction to maximal $f$-invariant domains should permit the description of  the dynamics of the complex Dulac germ on a standard quadratic domain. By comparison of Fatou coordinates on different sectors, we expect to obtain in the future the definition of Ecalle-Voronin moduli for Dulac germs. \\ 

An additional motivation for our study is to answer the following question: 
\begin{comment}
of embedding of a Dulac germ in a flow and other than giving a prerequisite for the analytic classification of the Dulac germ, we use the Fatou coordinate also in the upcomming preprint \cite{MRRZ2fractal},  related to fractal analysis of orbits. Here, the question is:
\end{comment}
 \emph{Can we \emph{recognize} a Dulac germ by looking at the size of $\varepsilon$-neighborhoods of its orbits?} A similar question, but for analytic germs, was discussed in \cite{resman} and \cite{nonlin}. The embedding in a flow of a Dulac germ $f$ will be used in \cite{MRRZ2fractal} to define an appropriate generalization of the length of the $\varepsilon$-neighborhood of an orbit of $f$.  The idea is to read the formal class and the Ecalle-Voronin moduli of a Dulac germ  in the \emph{parameter $\varepsilon$-space}, rather than in the phase space.\\

\begin{comment}
We consider here non-analytic (at $0$) Dulac germs. So far, the analytic classification for the Dulac germs is not known. Let $f\in C^\infty(0,d)$ be a Dulac germ and let $\tilde f$ be the complex extension of the Dulac germ to its standard quadratic domain from Ilyashenko \cite{}. recall that The sQD is the  domain at the origin of the complex extension of almost-regular germs by Ilyashenko (a spiraling domain at the oigin on the Riemann surface of the logarithm). In this paper, we construct the Fatou coordinate of $f$ in $C^\infty(0,d)$, but in the proof of the existence in Section~\ref{sec:fatou} we extend the construction of the Fatou coordinate to a small complex $\tilde f$-invariant domain around $(0,d)$. We believe the same construction will lead to finding the biggest $f$-invariant attractive and repulsive domains in the standard quadratic domain for the dynamics of $\tilde f$, and the construction of the complex Fatou coordinate for $\tilde f$ resp. $\tilde f^{-1}$ on respective domains.  The invariant domains correspond to petals on the unit disc in the regular case. The construction of the Fatou coordinate and invariant domains is based on the one described in e.g. \cite{loray} for regular germs.This is an important step toward analytic class of Dulac germs, following the standard Ecalle-Voronin idea.
\end{comment}

\paragraph{\bf{Organization of the paper}}

In Section 2 we formulate our main theorem about the existence and the uniqueness of the Fatou coordinate for a Dulac germ $f$ and its Dulac expansion $\widehat f$ respectively.

In Section \ref{sec:asymptotic_expansions},
%we notice that following Poincar{\'e} method to define the trans-asymptotic expansion in $\widehat{\mathfrak{L}}$ does not produce a well-defined expansion. 
we show the non-uniqueness of  transasymptotic expansions in the power-iterated log scale in general. 
In order to remedy this flaw, we introduce the notion of \emph{sectional asymptotic expansions} and define a particular type of such expansions adapted to  Dulac germs and their corresponding Fatou coordinates which will ensure their uniqueness.
The proof of the results of this section is given in Section \ref{sec:appendix}.

We recall in Section \ref{sec:embedding}
the classical notion of \emph{embedding as the time-one map
of a flow} and state the equivalence between the existence of a Fatou coordinate
and an embedding in a flow, for analytic germs on open intervals and for parabolic transseries. The proof of these facts is also postponed to Section \ref{sec:appendix}.

In Section~\ref{sec:examples} we give some examples of sectional asymptotic expansions.

Finally, Section \ref{sec:Fatou} is dedicated to the precise description of the Fatou coordinate of a Dulac germ and to the proof of the Theorem. The existence of a Fatou coordinate is established simultaneously for a Dulac germ and its formal expansion. This allows in particular to prove that the Fatou coordinate of a Dulac germ admits a sectional asymptotic expansion in a power-iterated log scale, in the sense of Section \ref{sec:asymptotic_expansions}. It is worth noticing that the proofs in Section~\ref{sec:Fatou} rely on the particular form of Dulac transseries. Part of the proof of the existence of the Fatou coordinate for Dulac germs in Section~\ref{sec:Fatou} is inspired by a similar classical result for parabolic analytic germs
which is explained, for example, in \cite{loray2}.

\section{Main definitions and results}

We define here the classes of \emph{power-iterated log transseries}, and recall the notions of \emph{Dulac germs} and \emph{Fatou coordinate}, which  are needed to state the Theorem about the Fatou coordinate of a Dulac germ. \\

We first introduce several classes of transseries.  We put $\boldsymbol{\ell}_{0}:=x$, $\boldsymbol{\ell}:=\boldsymbol\ell_1:=\frac{1}{-\log x}$, 
and define inductively $\boldsymbol{\ell}_{j+1}=\boldsymbol{\ell}\circ\boldsymbol{\ell}_{j}$,
$j\in\mathbb{N}$, as symbols for iterated logarithms. 

\begin{defi}[The classes $\widehat {\mathcal L}_j^\infty$ and $\widehat{\mathfrak L}$]\label{def:eljot}
Denote by $\widehat{\mathcal{L}}_{j}^{\infty}$, $j\in\mathbb N_0$, the set of all transseries of the
type: 
\begin{equation}\label{summable}
\widehat f(x)=
\sum_{i_{0}=0}^{\infty}\sum_{i_{1}=0}^{\infty}\cdots\sum_{i_{j}=0}^{\infty}a_{i_{0},\ldots, i_{j}}x^{\alpha_{i_{0}}}\boldsymbol{\ell}^{\alpha_{i_{0},i_{1}}}\cdots\boldsymbol{\ell}_{j}^{\alpha_{i_{0},\ldots,i_{j}}},\  a_{i_0,\ldots, i_j}\in\mathbb R,\ x>0,
\end{equation}
where $\left(\alpha_{i_{0},\ldots, i_{k}}\right)_{i_{k}\in\mathbb{N}}$
is a strictly increasing sequence of real numbers tending to $+\infty$ $($or finite$)$,
%\edz{I added or finite-am I right?}
for
every $k=0,\ldots,j$. If, moreover, $\alpha_{0}> 0$ $($the infinitesimal cases$)$, we denote the class by $\widehat{\mathcal L}_j$. The subset of $\widehat{\mathcal L}_1$ resp. $\widehat{\mathcal L}_1^\infty$ of transseries with only integer powers of $\boldsymbol\ell$ will be denoted by $\widehat{\mathcal L}$ resp. $\widehat{\mathcal L}^\infty$. 

\medskip
A \emph{monomial} in $\widehat{\mathcal L}_j^\infty$, $j\in\mathbb N_0,$
%\edz{Definition of a monomial in $\widehat{\mathcal L}_j$. Is notion 'term' good?}
is any term of the form $a x^{\gamma_0}\boldsymbol{\ell}^{\gamma_1}\cdots\boldsymbol{\ell}_{j}^{\gamma_j}$, $\gamma_i\in\mathbb R$, $i\in\{0,\ldots,j\}$, 
and $a\in\mathbb{R}\setminus\{0\}$.
\medskip

 A transseries $x^{\alpha_{i_{0}}} \big(\sum_{i_{1}=0}^{\infty}\cdots\sum_{i_{j}=0}^{\infty}a_{i_{0},\ldots, i_{j}}\boldsymbol{\ell}^{\alpha_{i_{0},i_{1}}}\cdots\boldsymbol{\ell}_{j}^{\alpha_{i_{0},\ldots, i_{j}}}\big)$, $i_0\in\mathbb N_0$, from \eqref{summable} is called a \emph{block} of $\widehat f\in\widehat {\mathfrak L}$. That is, a block of $\widehat f\in\widehat {\mathfrak L}$ is a transseries containing all monomials from $\widehat f$ sharing the same power of $x$.
 %\edz{Definition of a block.}
\medskip

For $\widehat f\in\widehat{\mathcal L}_j^\infty$, \emph{the support of $\widehat f$}, denoted by $\mathrm{Supp}(\widehat f)$, is defined as the set of exponents of all monomials in $\widehat f$ with non-zero coefficients:
%\edz{Added the def of the support.}
$$
\mathrm{Supp}(\widehat f):=\big\{(\alpha_{i_0},\alpha_{i_0, i_1},\ldots,\alpha_{i_0 ,i_1,\ldots, i_j}):\ a_{i_0,\ldots, i_j}\neq 0\big\}.
$$

Put
$$
\widehat{\mathfrak L}:=\bigcup_{j\in\mathbb N_0} \widehat{\mathcal L}_j^\infty
$$
for the class of all power-iterated logarithm transseries of finite depth in iterated logarithms.

\end{defi}

Note that the classes $\widehat{\mathcal{L}}_{j}^{\infty}$, $j\in\mathbb{N}_{0}$, are the sub-classes of power-iterated logarithm transseries, whose support
is any well-ordered subset of $\mathbb{R}^{j+1}$ (for the lexicographic
order). We restrict only to the subclass with exponents forming a strictly increasing sequence tending to $+\infty$. In this paper, we work with Dulac germs and their expansions, for which 
this condition is verified. 

Notice that $x=\boldsymbol{\ell}_{0}$. The classes $\widehat{\mathcal{L}}_{0}$ or $\widehat{\mathcal{L}}_{0}^{\infty}$
are made of formal power series: 
\[
\widehat{f}(x)=\sum_{i\in\mathbb N}a_{i}x^{\alpha_i},\ a_{i}\in\mathbb{R},\ x>0,
\]
such that $(\alpha_i)_{i}$ is a strictly increasing real sequence tending to $+\infty$. \\

For $\widehat f\in \widehat{\mathcal L}_j^\infty$, we denote by $\mathrm{Lt}(\widehat f)$ its \emph{leading term},
%\edz{definition of the leading term}
 which is defined as the smallest term $a_{\gamma_{0},\gamma_{1},\ldots,\gamma_{j}}\, x^{\gamma_{0}}\boldsymbol{\ell}^{\gamma_{1}}\boldsymbol{\ell}_{2}^{\gamma_{2}}\cdots\boldsymbol{\ell}_{j}^{\gamma_{j}}$ in $\widehat f$ (by the lexicographic order on the monomials) with a non-zero coefficient $a_{\gamma_{0},\gamma_{1},\ldots,\gamma_{j}}\neq 0$. The tuple $(\gamma_0,\gamma_{1},\ldots,\gamma_{j})$ is called \emph{the
order} of $\widehat f$, and is denoted by $\text{ord}(\widehat f)=(\gamma_0,\gamma_{1},\ldots,\gamma_{j})$. The transseries $\widehat f\in\widehat{\mathcal{L}}_{j}$, $j\in\mathbb N_0$, is called \emph{parabolic} if $\text{ord}(\widehat f)=(1,0,\ldots,0).$\\

We denote by $\mathcal G$ the set of all germs at $0$ of real functions
%\edz{Germs explained.}
defined on some open interval $(0,\varepsilon)$, $\varepsilon>0$ (meaning that two functions $f$ and $g$ define the same germ if there exists an interval $(0,\varepsilon)$, $\varepsilon>0$, where they coincide). Furthermore, by $\mathcal G_{AN}\subset \mathcal{G}$ we denote the set of all germs at $0$ of real functions defined and \emph{analytic} on some interval $(0,\varepsilon),\ \varepsilon>0$.
\smallskip

%\edz{The explanation added.}
In the paper, we will use notation $o(\cdot)$ in two different contexts: for germs from $\mathcal G$ and for transseries from $\widehat{\mathfrak L}$. We always mean that $x\to 0$. In the case of germs, $f(x)=o(x^{\gamma_0}\boldsymbol\ell^{\gamma_1}\cdots\boldsymbol\ell_k^{\gamma_k})$ means that $\lim_{x\to 0}\frac{f(x)}{x^{\gamma_0}\boldsymbol\ell^{\gamma_1}\cdots\boldsymbol\ell_k^{\gamma_k}}=0$. In the case of transseries, $\widehat f(x)=o(x^{\gamma_0}\boldsymbol\ell^{\gamma_1}\cdots\boldsymbol\ell_k^{\gamma_k})$ means that the leading monomial of $\widehat f$ is lexicographically of strictly bigger order than the monomial $x^ {\gamma_0}\boldsymbol\ell^{\gamma_1}\cdots\boldsymbol\ell_k^{\gamma_k}$.
\medskip

\begin{defi}\label{def:fgs} 
%\edz{Definition added. Pavao-the name?}
Let $(\alpha_i)_i$ be a sequence of strictly positive real numbers. We say that $(\alpha_i)_i$ is an \emph{increasing sequence of finite type} if it is of one of the following types:
\begin{enumerate}
\item finite and strictly increasing with $i$, or

\item infinite, strictly increasing as $i\to \infty$ and \emph{finitely generated} (there exist strictly positive real generators $\beta_1,\ldots,\beta_n$, such that for every $\alpha_i$ there exist $a_1^i,\ldots,a_n^i\in\mathbb N$ such that $\alpha_i=\sum_{j=1}^{n}a_j^i \beta_j$).
\end{enumerate}
Note that in the case $(2)$ it necessarily follows that $\alpha_i\to\infty$, as $i\to\infty$.
\end{defi}

We now recall the definition of a \emph{Dulac series} from \cite{Dulac}, \cite{ilya} or \cite{roussarie}, and define what we mean by a \emph{Dulac germ}.

\begin{defi}[Dulac germs]\label{def:dulac_germ} \
	
	1. We say that $\widehat{f}\in\widehat{\mathcal{L}}$ is a \emph{Dulac
		series} $($\cite{Dulac}, \cite{ilya}, \cite{roussarie}$)$ if it is of the form: 
	\begin{equation}\label{eq:dulacc}
	\widehat{f}=\sum_{i=1}^{\infty}P_{i}(\boldsymbol{\ell})x^{\alpha_{i}},
	\end{equation}
	where $(P_{i})_i$ is a sequence of polynomials and $(\alpha_{i})_{i}$ an  increasing sequence of finite type (see Definition~\ref{def:fgs}).
	
	2. We say that $f\in \mathcal G_{AN}$ is a \emph{Dulac germ} if:
	\begin{itemize}
		\item there exists a sequence $\left(P_{i}\right)$ of \emph{polynomials} and  an increasing sequence of finite type $(\alpha_{i})_{i}$ (see Definition~\ref{def:fgs}), such that
		$$
		f-\sum_{i=1}^{n} P_i(\boldsymbol\ell) x^{\alpha_i}=o(x^{\alpha_{n}}),\ n\in\mathbb N,
		$$
		\item $f$ is \emph{quasi-analytic}: it can be extended to an analytic
		%\edz{removed bounded, added with the expansion, as in Ilyashenko}
		function to a standard quadratic domain in $\mathbb C$ with the same expansion \eqref{eq:dulacc}, as precisely defined by Ilyashenko, see \cite{ilya}, \cite{roussarie}.
	\end{itemize}
	If moreover $P_1\equiv1$, $\alpha_1=1$, and at least one of the polynomials $P_i$, $i>1$, is not zero, then $f$ is called a \emph{parabolic Dulac germ}. 
\end{defi}

The quasi-analyticity property ensures that a Dulac germ $f$ is uniquely determined by its Dulac asymptotic series $\widehat f$, see \cite{ilya}. 
%\edz{added}

Note that the germs of first return maps of hyperbolic polycycles of planar analytic vector fields are Dulac germs in the sense of  Definition \ref{def:dulac_germ}, see e.g. \cite{Dulac,ilya}.  

\smallskip
We recall finally what is the Fatou coordinate of a real germ:

\begin{defi}[Fatou coordinate]\label{def:ffatou}~

1. Let $f$ be an analytic germ on $(0,d)$, $d>0$. Let $(0,d)$ be invariant for $f$. We say that a strictly monotonic
analytic germ $\Psi$ on $(0,d)$ is a \emph{Fatou coordinate} for
$f$ if 
\begin{equation}
\Psi(f(x))-\Psi(x)=1,\ x\in(0,d).\label{eq:fat0}
\end{equation}

2. Let $\widehat{f}\in\widehat{\mathcal{L}}$ be parabolic. We say
that $\widehat{\Psi}\in\widehat{\mathfrak{L}}$ is a \emph{formal
Fatou coordinate} for $\widehat{f}$ if the following equation is
satisfied formally in $\widehat{\mathfrak{L}}$ : 
\begin{equation}
\widehat{\Psi}(\widehat{f})-\widehat{\Psi}=1.\label{eq:rfatou}
\end{equation}
\end{defi}
\noindent Classically, equation \eqref{eq:fat0} is  called the \emph{Abel equation} for $f$.

\bigskip

We now formulate the main result of this paper. The notion of \emph{sectional asymptotic expansions} is introduced in Section~\ref{sec:asymptotic_expansions} to ensure the uniqueness of transfinite asymptotic expansions in $\widehat{\mathfrak {L}}$. In particular, the notion of \emph{integral sections} is adapted to the Fatou coordinate.

The constructive proof of the Theorem and a more precise description of the Fatou coordinate for a Dulac germ is given in Section~\ref{sec:Fatou}.

\begin{thmC} Let $f\in \mathcal G_{AN}$ be a parabolic
Dulac germ and let $\widehat{f}\in\widehat{\mathcal{L}}$ be its Dulac
expansion. 

1. There exists a unique $($up to an additive constant$)$ formal Fatou
coordinate $\widehat{\Psi}$ for $\widehat{f}$ in $\widehat{\mathfrak L}$. It belongs to $\widehat{\mathcal L}_2^\infty.$

2. There exists a unique $($up to an additive constant$)$ Fatou coordinate
$\Psi\in \mathcal G_{AN}$ for the germ $f$ which admits a sectional asymptotic expansion with respect to an integral section
%\edz{added integral}
in the class $\widehat{\mathfrak L}$ $($in the sense of Definitions~\ref{def:sectional} and \ref{def:is}$)$.

3. Let $\mathbf s$ be a fixed integral section and  $\widehat \Psi\in\widehat {\mathcal L}_2^\infty$ the formal Fatou coordinate $($with a fixed choice of the additive constant$)$. Then there exists a choice of the additive constant in $\Psi\in\mathcal G_{AN}$ from 2. such that the formal Fatou coordinate $\widehat\Psi$ is the $($unique$)$ sectional asymptotic expansion of $\Psi\in\mathcal G_{AN}$ with respect to $\mathbf s$. The Fatou coordinate $\Psi$ is of the
form: 
\[
\Psi=\Psi_{\infty}+R,
\]
where $\Psi_{\infty}\to\infty$ and $R=o(1)$, as $x\to 0$.

4. $\widehat{\Psi}=\widehat{\Psi}_{\infty}+\widehat{R}$, where $\widehat{\Psi}_{\infty}\in\widehat{\mathcal{L}}_{2}^{\infty}$ is the sectional asymptotic expansion of $\Psi_{\infty}$ with respect to $\mathbf s$
and $\widehat{R}\in \widehat{\mathcal{L}}$ is the
sectional asymptotic expansion of $R$ with respect to $\mathbf s$.
\end{thmC}
Note that different choices of integral sections $\mathbf s$ in \emph{3.} lead to change $\Psi$ only by an additive constant $C\in\mathbb R$.
\smallskip

Note that $\Psi_{\infty}\to\infty$, $x\to0,$ is the \emph{infinite}
part, and $R=o(1)$, $x\to0$, is the \emph{infinitesimal} part. We
call $\Psi_{\infty}$ the \emph{principal part} of $\Psi$.
\medskip
\begin{obs}\label{cor:efatou} Let $f\in\mathcal G_{AN}$ be a parabolic Dulac germ. Let 
$\mathrm{ord}(\mathrm{id}-f)=(\alpha_{1},m)$, $m\in\mathbb{N}_{0}^{-}$, $\alpha_1>1$.
The function $R$ from the Theorem satisfies the
\emph{modified Abel} difference equation: 
\begin{equation}
R(f(x))-R(x)=\delta(x).\label{eq:modi}
\end{equation}
Here, $\delta$ is analytic on an open interval  $(0,d)$, $d>0$, and \emph{small}: $\delta(x)=O(x^{\gamma})$,
with $\gamma>\alpha_{1}-1$. \end{obs} 

\begin{obs}[\emph{Non-uniqueness} of the Fatou coordinate in $\mathcal G_{AN}$, if one does not require the existence of its asymptotic expansion in $\widehat{\mathfrak L}$]\ 

Note that any strictly monotonic $\Psi\in\mathcal G_{AN}$ whose inverse $\Psi^{-1}$, as a germ at infinity, satisfies $\Psi^{-1}(w+1)=f(\Psi^{-1}(w))$, is a Fatou coordinate for $f$. This gives us freedom of choice of $\Psi^{-1}$ on the fundamental domain $[0,1)$ and the rule for its extension at the neighborhood of $\infty$, thus, non-unicity of a Fatou coordinate for $f$. 

In particular, let $\Psi_1\in\mathcal G_{AN}$ be the Fatou coordinate constructed in the proof of the Theorem that admits a sectional expansion $\widehat \Psi_1 \in \widehat{\mathcal L}_2^\infty$. Let $\Psi_2\in\mathcal G_{AN}$ be defined by 
$\Psi_2:=\Psi_1+T_1\circ \Psi_1, $
where $T_1$ is any periodic function on $\mathbb R$ of period $1$ whose derivative $T_1'$ is bounded in $(-1,1)$ (e.g. $T_1(x)=\frac{1}{4\pi}\sin (2\pi x),\ x\in\mathbb R$).
It can be easily checked that $\Psi_2$ is also a Fatou coordinate for $f$ (by Definition~\ref{def:ffatou}). It does not admit an expansion in $\widehat{\mathfrak L}$, due to periodicity of $T_1$.
\end{obs}

\begin{obs}\label{rem:non}
%\edz{ADDED!}
Note that the Fatou coordinate $\Psi\in\mathcal G_{AN}$ for Dulac germ $f$ in the Theorem admitting only a \emph{sectional} asymptotic expansion in $\widehat{\mathfrak L}$ is \emph{not} unique, as we will show in Example~\ref{ex:add} in Subsection~\ref{sec:subthree}. On the other hand, we prove in Subsection~\ref{sec:subthree} that a Fatou coordinate  $\Psi\in\mathcal G_{AN}$ for $f$ admitting a sectional asymptotic expansion with respect to an \emph{integral sectional} in $\widehat{\mathfrak L}$ is unique.
\end{obs}

\section{Sectional asymptotic expansions}
\label{sec:asymptotic_expansions}This paper is motivated by the study
of Dulac germs (see Definition~ \ref{def:dulac_germ}) and their Fatou coordinates. The Dulac asymptotic
expansions involve monomials of the form $x^{\alpha}\log^{p}x$, $\alpha\in\mathbb{R}$,
$p\in\mathbb{N}$, where each power of $x$ is multiplied by a \emph{polynomial} in $\log x$. 
The Dulac expansion of a germ is therefore uniquely given by the classical Poincar\'e algorithm. 
On the contrary, it turns out that the \emph{asymptotic expansions}
of Fatou coordinates of Dulac germs involve also \emph{powers of $x$ multiplied by possibly divergent series in iterated
logarithms}. Hence, the classical Poincar\'e algorithm does not suffice to produce them. In this work we give a generalization
of the algorithm to explain what it means for a germ from $\mathcal{G}$ to have a transfinite asymptotic expansion. In particular, our construction applies to Fatou coordinates.

Let us first illustrate our definitions on some examples. 

\begin{example}
%	\edz{JP, the example added}
1) Suppose that we want to express that the asymptotic expansion of the germ $f\in\mathcal{G}$ is the series
$\widehat{f}\left(x\right)=x\left(1+\boldsymbol{\ell}+\boldsymbol{\ell}^{2}+\cdots\right)+x^{2}$.
Obviously, we first require that, for every $p\in\mathbb{N}$, $f\left(x\right)-\sum_{n=0}^{p}x\boldsymbol{\ell}^{n}=o\left(x\boldsymbol{\ell}^{p}\right)$. This corresponds to the first
steps of the Poincar{\'e} algorithm, which are indexed by integers. Now an extra step is needed, which may be thought of as indexed by the ordinal number $\omega$. To perform this step, we can take advantage of the convergence of the series $\sum_{n\ge0}\boldsymbol{\ell}^{n}=\frac{1}{1-\boldsymbol{\ell}}$
to require that $f\left(x\right)-x\frac{1}{1-\boldsymbol{\ell}}$
is equivalent to $x^{2}$ at the origin.\\

2) The above process does not work if the first power
of $x$ in the series $\widehat{f}$ is multiplied by a \emph{divergent
	series} in $\boldsymbol{\ell}$. For example, how to give a meaning to the statement that the series $\widehat{f}\left(x\right)=x\sum_{n\ge0}n!\boldsymbol{\ell}^{n}+x^{2}$
is an asymptotic expansion of the germ $f\in\mathcal{G}$?
%\edz{I changed the interrogative phrase.}
Our answer
consists in choosing a germ $g\in\mathcal{G}$ to which the series
$x\sum_{n\ge0}n!\boldsymbol{\ell}^{n}$ is asymptotic \emph{in
	the classical sense}. The germ $g$ can be seen as a \emph{sum} of this series. Then we perform an extra step by requiring that $f\left(x\right)-g\left(x\right)$
is equivalent to $x^{2}$ at the origin. That is, we first follow the usual Poincar\'e algorithm along steps indexed by integers. Once we have reached the step indexed by the first \emph{limit ordinal} $\omega$, we associate a \emph{sum} $g$ to (in general divergent)
%\edz{I put in brackets (in general divergent)-it does not have to be}
 series $x\sum_{n\ge0}n!\boldsymbol{\ell}^{n}$ in order to proceed further.\\

3) In the same way, we say that a germ $f\in\mathcal{G}$
admits an
%\edz{I would say an asy exp, not the } 
asymptotic expansion $\widehat{f}\left(x\right)=x\sum_{n\ge0}n!\boldsymbol{\ell}^{n}+x^{2}\sum_{n\ge0}\left(n!\right)^{2}\boldsymbol{\ell}^{n}+x^{3}$
if there exist two germs $g_{1},g_{2}\in\mathcal{G}$ (two \emph{sums}) such that:
\begin{itemize}
 \item[-] $x\sum_{n\ge0}n!\boldsymbol{\ell}^{n}$ is the (classical Poincar\' e) asymptotic expansion of $g_{1}$ at $0$, 

 \item[-] $x^{2}\sum_{n\ge0}\left(n!\right)^{2}\boldsymbol{\ell}^{n}$ is the (classical) asymptotic expansion of $g_2$ at $0$, and

 \item[-] $f-g_1-g_2$ is equivalent to $x^{3}$ at the origin. 
\end{itemize}

\noindent In this case, we will say that $x\sum_{n\ge0}n!\boldsymbol{\ell}^{n}+x^{2}\sum_{n\ge0}\left(n!\right)^{2}\boldsymbol{\ell}^{n}$ is a \emph{transfinite} asymptotic expansion of $g_1+g_2$ and that $\widehat f$ is a \emph{transfinite} asymptotic expansion of $f$.
%\edz{Reformulated. Also: transfinite expansion instead of generalized.}
\end{example}

\begin{obs}[Non-uniqueness of asymptotic expansions in $\widehat{\mathfrak L}$ of germs from $\mathcal G$]\label{nonuniq}In Subsection~\ref{subsec:tP} below,
%	\edz{I added the exact reference. Slightly reformulated.}
we give a general version of the method illustrated by the above examples. The usual algorithm due to Poincar\'e, which associates an asymptotic expansion to a germ, proceeds term by term along steps indexed by integers, that is, by ordinals less than $\omega$. For germs of functions considered in this work, \emph{transfinite} asymptotic expansions will be produced in $\widehat{\mathfrak{L}}$ by a \emph{transfinite version} of this algorithm. The algorithm continues along steps indexed by ordinals (bigger than $\omega$), as it is the case in the examples above. We see that, when we reach a step indexed
	by a \emph{limit ordinal}, we have to provide a \emph{sum} in order to continue.
	
This feature leads us to stress an important fact. Whereas the classical algorithm associates to a germ a well-defined, unique asymptotic expansion, the expansions produced in $\widehat{\mathfrak{L}}$ by the generalized method are in general not unique: in the examples above, different choices of \emph{sums} $g$, $g_1$ and $g_2$ may lead to different asymptotic expansions.  This non-uniqueness is illustrated by Example~\ref{ex:nonun} below.
\end{obs}

%\medskip
\subsection{Transfinite Poincar\' e algorithm.}\label{subsec:tP}

Let the classes $\widehat{\mathcal L}_j^\infty$, $j\in\mathbb N$, and $\widehat{\mathfrak L}$ be as in Definition~\ref{def:eljot}. In this section we define in full generality what it means for a series
$\widehat{f}\in\widehat{\mathfrak{L}}$ to be an asymptotic expansion
of a germ $f\in\mathcal{G}$. This definition can be seen as \emph{transfinite}
version of the usual definition of Poincar\'e.
\medskip

%Let $f\in \mathcal G$. We construct an asymptotic expansion of $f$ in $\widehat {\mathfrak{L}}$ of some ordinal order type $\theta$,  following the algorithm of Poincar\' e. To this end, we construct  a transfinite sequence of \emph{partial expansions} $(\widehat f_{\nu})_{0\leq \nu<\theta}$ and a transfinite sequence of auxiliary cut-off germs $(f_{\nu})_{0\leq \nu<\theta}$, $f_{\nu}\in\mathcal G$,  indexed by ordinals $\nu$ strictly smaller than $\theta$,\edz{Added 'indexed by ordinals...', referee's comment $\nu\geq 0$. I have in the end put $0\leq \nu$ in the indexing, since maybe it is not clear that we really start at the first odinal.} following the algorithm:

Consider a germ $f\in\mathcal{G}$, a transseries $\widehat{f}\in\widehat{\mathfrak{L}}$
and an ordinal $\theta\ge 1$ (where $0$ denotes the smallest ordinal). We say that $\widehat{f}$
	is a \emph{truncated asymptotic expansion of length $\theta$ of $f$},
and we write $f\underset{\theta}{\sim}\widehat{f}$, if there exist
a sequence $\left(h_{\alpha}\right)_{0\leq \alpha<\theta}$ of elements
of $\mathcal{G}$, and a sequence $(\widehat{f}_{\alpha})_{0\leq \alpha<\theta}$
of elements of $\widehat{\mathfrak{L}}$, such that:
\medskip
\begin{enumerate}[1., font=\textup, nolistsep, leftmargin=0.6cm]
	\item $\widehat{f}=\lim_{\nu<\theta}\widehat{f}_{\nu}$ (for
	the product topology with respect to the discrete topology introduced
	in \cite{mrrz2});
	\item $h_{0}=f$, $\widehat{f}_{0}=0$;
	\item for all $0<\alpha<\theta$,
%	\edz{Added $1\leq \alpha$, as referee wanted before.}
	we have:
	\begin{enumerate}[leftmargin=0.3cm]
		\item if $\alpha:=\nu+1$ is a successor ordinal, then
		\begin{enumerate}[leftmargin=0.1cm]
			\item either $h_{\nu}\rightarrow0$ faster than \emph{any}
			monomial from $\widehat{\mathfrak{L}}$ and $\theta=\alpha+1$, $\widehat f_\alpha=\widehat{f}_{\nu}$
			and $h_{\alpha}=h_{\nu}$,
			\item or there exists a monomial in $\widehat{\mathfrak L}$ (see Definition~\ref{def:eljot}), denoted by $\mathrm{Lt}\left(h_{\nu}\right)\in\widehat{\mathfrak{L}}$,
			such that
			\[
			\lim_{x\rightarrow0}\frac{h_{\nu}\left(x\right)}{\mathrm{Lt}\left(h_{\nu}\right)}=1,
			~\widehat{f}_{\alpha}=\widehat{f}_{\nu}+\mathrm{Lt}\left(h_{\nu}\right),~h_{\alpha}=h_{\nu}-\mathrm{Lt}\left(h_{\nu}\right),
%		\edz{$h_\nu$ instead of $f_\nu$}
				\]
		\end{enumerate}
		\item if $\alpha<\theta$ is a limit ordinal, then  $\widehat{f}_{\alpha}=\lim_{\nu<\alpha}\widehat{f}_{\nu}$ and
		\begin{enumerate}[leftmargin=0.1cm]
		\item either $\theta=\alpha+1$ and, for every $\beta\in\mathbb{R}$, there exists a block of $\widehat{f}_{\alpha}$ whose monomials are smaller than $x^{\beta}$, 
		\item
		or there exists a germ $g_{\alpha}\in\mathcal{G}$ with 
		\[
		g_{\alpha}\underset{\alpha}{\sim}\widehat{f}_\alpha\text{ and }h_{\alpha}=f-g_{\alpha}. 
		\]
	    \end{enumerate}
%		\edz{Should here be $\widehat f_\alpha$ instead of $\widehat f$? i believe both are true.}
	\end{enumerate}
\end{enumerate}

Notice that in case (b), the sequence $(\widehat f_{\nu})_{\nu<\alpha}$ necessarily converges in the product topology with respect to the discrete topology introduced in \cite{mrrz2}, due to the fact that the orders of the monomials in the expansion strictly increase in each step.
%\edz{Added the explanation why the limit always exists. Is it too much? Do you agree?}

%\noindent Let $g_\alpha\in\mathcal G$ be a germ admitting the transfinite Poincar\' e expansion $\widehat f_\alpha$. Then put:
%$$
%f_\alpha=f-g_\alpha.
%$$ 
Moreover,
%\edz{Shortened. Is it ok like this, I think this explanation is sufficient after the algo is changed.}
%\edz{Added the sentence explaining the existence. Agree? It is almost tautology. Is it understandable?}
the existence of such a germ $g_\alpha$ is trivially guaranteed. %Indeed, let $\alpha_p$  denote the limit ordinal preceding $\alpha$ (or $\alpha_p=0$ if it does not exist). Then e.g. $h_{\alpha_p}$ is one germ whose Poincar\' e expansion is $\widehat f_{\alpha}-\widehat f_{\alpha_p}$. W
Indeed, without any additional restrictions imposed on the choice, we always have a trivial choice of the germ: if $f\underset{\theta}{\sim}\widehat{f}$, then $f$ itself admits an asymptotic expansion $\widehat f_\alpha$ of length $\alpha$, $f\underset{\alpha}{\sim}\widehat{f}_\alpha$, for any $\alpha<\theta$. Of course, we may (and usually do) choose some other germ that admits the same expansion. In particular, a good choice of a germ will be given by an integral section (see Section~\ref{sec:is}).\\
%\end{enumerate}

Finally, in the cases 3.(a)(i) and 3.(b)(i) of the above definition,  the series $\widehat{f}_{\alpha}$ is considered as a \emph{total expansion} of $f$, rather than a \emph{truncated expansion}. We call $\widehat{f}_{\alpha}$  simply an \emph{asymptotic expansion of $f$}, and we write $f\sim\widehat{f}_{\alpha}$, without any reference to its length.\\

Actually, the germs $g_{\alpha}$ above can be seen as \emph{sums} of the transseries $\widehat{f}_{\alpha}$. 
%Notice that Step 2.(i) of the algorithm is well-defined by the Poincar\' e method. However, Step 2.(ii) of the algorithm is not well-defined, due to a non-unique choice of the germ $g_\alpha\in\mathcal{G}$, by adding of exponentially small terms.
 Since $\boldsymbol\ell_k$ are exponentially small with respect to $\boldsymbol\ell_{k+1}$, $k\in\mathbb N_0$, as $x\to 0$, because of
 the relation 
 $$e^{-\frac{1}{\boldsymbol\ell_{k+1}}}=\boldsymbol\ell_{k},$$
 different choices of sums $g_\alpha$ lead in general to different transfinite asymptotic expansions.  This is illustrated in Example~\ref{ex:nonun} below.

\begin{example}\label{ex:nonun} \ 
%	\edz{THE WHOLE SECTION 3.2. and 3.3. CHANGED SO THAT WE CAN IMPOSE LINEARITY IN DEFINITION of sections and this does not violate integral sections. Gain: definition of section follows directly the algorithm described before (see Def.3.3 and the text after it), restricted to an appropriate subset which is well-adapted to a Fatou coordinate integral sections are linear (see Def. 3.14 and comment after), minus: Def. 3.14 of integral sections more complicated since we have to sum at limit ordinals of the original series, not at log levels above and then multiply by $x$ as in definition before.}

\begin{enumerate}
\item $
f(x)=\frac{x}{1-\boldsymbol\ell}+x^2.
$
We can also write it as $f(x)=x(\frac{1}{1-\boldsymbol\ell}+e^{-\frac{1}{\boldsymbol\ell}})$.
The algorithm described above produces the intermediate series $\widehat{f}_{\omega}=x\left(1+\boldsymbol{\ell}+\boldsymbol{\ell}^{2}+\cdots\right)$.
Now we can associate two sums to $\widehat{f}_{\omega}$, namely:
\[
g_{\omega,1}\left(x\right)=\frac{x}{1-\boldsymbol{\ell}}\text{ and }g_{\omega,2}\left(x\right)=\frac{x}{1-\boldsymbol{\ell}}+x^{2}=f\left(x\right).
\]
Note that the second sum corresponds to the ending remark in the algorithm, which
says that $f$ itself can always be chosen as a possible sum for any of the
intermediate series $\widehat{f}_{\alpha}$, for $\alpha<\theta$ limit ordinal.
%\edz{added limit ordinal. This is I think important.}
Hence, depending on the choice made at stage $\omega$, the series:
$$\widehat f(x)=x(1+\boldsymbol\ell+\boldsymbol\ell^2+\boldsymbol\ell^3+\cdots),\ \text{and } \widehat f(x)=x(1+\boldsymbol\ell+\boldsymbol\ell^2+\cdots)+x^2,$$
are two (equally good) asymptotic expansions of $f$ in $\widehat{\mathcal L}\subset \widehat{\mathfrak{L}}$.

\item
$
f(x)=\frac{x}{1-\boldsymbol\ell}.
$
Obviously, $f$ admits an asymptotic expansion $\widehat f\in\widehat{\mathcal L}_1$, $\widehat f=x(1+\boldsymbol\ell+\boldsymbol\ell^2+\cdots)$. However, equally legitimate asymptotic expansions by the transfinite Poincar\' e algorithm from Remark~\ref{nonuniq} are, for example:
\begin{align*}
f(x)=x\big(\frac{1}{1-\boldsymbol\ell}+\boldsymbol\ell_k e^{-\frac{1}{\boldsymbol\ell}}\big)&-x^2\boldsymbol\ell_k\\
&\Rightarrow\widehat f(x)=x(1+\boldsymbol\ell+\boldsymbol\ell^2+\cdots)-x^2\boldsymbol\ell_k\ \in\widehat{\mathcal{L}}_k,\ k\in\mathbb N.
\end{align*}
\end{enumerate}
In this manner, we can easily generate non-unique asymptotic expansions of $f\in \mathcal G$ in $\widehat{\mathfrak{L}}$. Moreover, they can belong to any $\widehat{\mathcal L}_k$, $k\in\mathbb N$. Note additionally that there does not exist a minimal $j\in\mathbb N$ such that the expansion of $f$ in $\widehat{\mathcal L}_j$ is unique. Indeed, if we put $\boldsymbol\ell$ instead of $\boldsymbol\ell_k$ in the example above, we get non-uniqueness of the expansion already in the class $\widehat{\mathcal L}_1$. 

On the contrary, just note that in the class $\widehat{\mathcal L}_0^\infty$, the asymptotic expansion of a germ, if it exists, is unique (by the Poincar\' e algorithm). 
%\edz{Added a sentence.}
%\edz{Last comment about minimal $j$ added.}
\end{example}

%Finally, we call \emph{the order of $f\in\mathcal{L}_{tot}^{\infty}$} the lowest of orders of all its asymptotic expansions. We denote it by $\text{ord}(f)$.

\subsection{Sections. Sectional asymptotic expansions.}
%\edz{Section mappings to sections.}
 %To make the limit-ordinal step in the transfinite Poincar\' e algorithm from Subsection~\ref{subsec:tP} well-defined and unique, in Definition~\ref{def:sectional} below we introduce the notion of \emph{sections} and \emph{sectional asymptotic expansions}. The goal is to ensure injectivity of the mapping that assigns to a germ at $0$ (defined on some interval $(0,\varepsilon)$, $\varepsilon>0$)\edz{I believe this sentence was not correct, it is this way?} its asymptotic expansion on all limit ordinal levels except possibly the final. 

Based on the content of the previous section, we can now define the notion of \emph{asymptotic transseries}:

%First, let $\mathcal S_0\subset \mathcal G$ be the set of all germs from $\mathcal G$ that admit a power asymptotic\edz{Changed section functions to sections.} expansion in $\widehat {\mathcal L}_0^\infty$. This expansion is obviously unique. That is, $f\in \mathcal S_0$ if there exists a strictly increasing sequence $(\alpha_i)_{i\in\mathbb N}\subset\mathbb R$ tending to $+\infty$ and a sequence of reals $(a_i)$, such that
%$$
%f(x)-\sum_{i=1}^{n} a_i x^{\alpha_i}=o(x^{\alpha_n}),\ n\in\mathbb N.
%$$
%Equivalently, $f\in\mathcal S_0$ if it admits a Poincar\' e asymptotic expansion: $$\widehat f(x)=\sum_{i\in\mathbb N}a_i x^{\alpha_i}\in 
%\widehat {\mathcal L}_0^\infty,\ x\to 0.$$
%Denote by $\widehat{\mathcal S}_0\subset \widehat{\mathcal L}_0^\infty$ the set of all power asymptotic expansions of germs from $\mathcal S_0$.

\begin{defi}[The set $\widehat{\mathcal S}$ of asymptotic transseries in $\widehat{\mathfrak L}$]\label{sectional}
A series $\widehat{f}\in\widehat{\mathfrak{L}}$ such
that there exists an element $f\in\mathcal{G}$ with $f{\sim}\widehat{f}$ $($as described in Section
\ref{subsec:tP}$)$ is called an \emph{asymptotic transseries}.
%\edz{I put everywhere $\subseteq$, not to imply that there exist germs that do not expans or series from $\widehat{\mathfrak L}$ which are not expansions of any germ.}
We denote by $\widehat{\mathcal{S}}\subseteq\widehat{\mathfrak{L}}$
the set of all asymptotic transseries. By $\mathcal{S}\subseteq\mathcal{G}$, we denote
the set of all germs which admit an asymptotic expansion in $\widehat{\mathcal{S}}$ $($in the sense of Subsection~\ref{subsec:tP}$)$.
\end{defi}
%Analogously, we define the set
%$$
%\mathcal S=\cup_{j\in\mathbb N_0} \mathcal S_j\subset \mathcal G
%$$
%as the set of all germs from $\mathcal G$ that admit an asymptotic expansion in $\widehat{\mathfrak L}$.
%\medskip

We have stressed in the previous section the non-uniqueness of asymptotic expansions of germs in $\mathcal{S}$. This comes from the multiple ``sums'' which can be associated to the intermediate series at each limit ordinal step of the transfinite Poincar{\'e} algorithm. Therefore, in Definition~\ref{sectional}, asymptotic expansions of germs in $\widehat{\mathfrak L}$ (as well as in any $\widehat{\mathcal L}_j^\infty$, $j\geq 1$) are not unique. Note that the power asymptotic expansion of a germ (that is, the asymptotic expansion in $\widehat{\mathcal L}_0^\infty$), if it exists, is by classical Poincar\' e algorithm always unique. 
%\edz{Added this sentence about power expansions. We did not put it anywhere.}
The uniqueness of the transfinite asymptotic expansion can be obtained by setting a ``summation mapping'' which associates to every series $\widehat{f}_{\alpha}$ indexed by limit ordinal $\alpha$ a well
defined sum $g_\alpha\in\mathcal G$ (using the notations of Section \ref{subsec:tP}).

Actually, picking a particular ``sum'' of a series among infinitely many possible ones is an operation quite similar to 
picking an element in the fiber over a point of the base space, in the
situation of a fiber bundle. Hence in the sequel we call such a summation mapping a \emph{section} 
(see the precise Definition \ref{def:sectional}), and the associated expansions \emph{sectional asymptotic expansions}.
In particular, we define later a particular choice of a section adapted to our problem that will guarantee the uniqueness of the asymptotic expansion of a germ with respect to this section.
%\edz{I reformulated the phrase. we pick one special section. It seems to me that the fact that the asy. exp. is unique after choosing a particular section is obvious-we do not prove that.}
\\

It is convenient to impose upon our sections to be linear maps defined on appropriate vector subspaces of $\widehat{\mathcal{S}}$. Notice however that, while $\widehat{\mathcal{S}}$ is obviously a vector space, it is not the case for the set $\mathcal S$, as the following example shows. 

\begin{example}[$\mathcal S$ is not a vector space] Take $$f_1(x)=x\frac{1}{1-\boldsymbol\ell}\in\mathcal G,\ f_2(x)=x\frac{1}{1-\boldsymbol\ell}+x^2\sin(x)\in\mathcal G$$ Note that $f_2$ can be written as $f_2(x)=x\big(\frac{1}{1-\boldsymbol\ell}+e^{-1/\boldsymbol\ell}\sin(e^{-1/\boldsymbol\ell})\big)$. Since both germs $\frac{1}{1-\boldsymbol\ell}$ and $\frac{1}{1-\boldsymbol\ell}+e^{-1/\boldsymbol\ell}\sin(e^{-1/\boldsymbol\ell})$ admit $1+\boldsymbol\ell+\boldsymbol\ell^2+\ldots\in\widehat{\mathcal L}_0$ as a power asymptotic expansion, both $f_1$ and $f_2$ admit $\widehat f=x(1+\boldsymbol\ell+\boldsymbol\ell^2+\ldots)\in\widehat{\mathcal L}_1$ as their asymptotic expansion in $\widehat{\mathfrak L}$ (as in Definition~\ref{sectional}). On the other hand, $f_1(x)-f_2(x)=x^2\sin x$ obviously does not admit an asymptotic expansion in $\widehat{\mathfrak L}$.
\end{example}
\begin{defi}\label{def:sectional} \

$(i)$ A \emph{section} $\mathbf s$ is any linear mapping defined on some vector space $\widehat{\mathcal I}_{\mathbf s}\subseteq \widehat{\mathcal S}$, $$\mathbf s:\widehat{\mathcal I}_{\mathbf s}\to \mathcal G_{AN},$$
where $\widehat g\in\widehat{\mathcal S}$ is an asymptotic expansion of $\mathbf s(\widehat g)\in\mathcal G_{AN}$ $($obtained by the transfinite Poincar\' e algorithm from Section~\ref{subsec:tP}$)$.
%\edz{Added.}

We say that the germ $g=\mathbf{s}(\widehat g)$ is the \emph{sectional sum of $\widehat g$ with respect to the section $\mathbf s$}.
\smallskip

$(ii)$ We say that a germ $f\in \mathcal G_{AN}$ admits a \emph{sectional asymptotic expansion $\widehat f\in\widehat{\mathcal S}$ with respect to the section $\mathbf s$} if:
\begin{itemize}
\item[-] $f\sim \widehat f$ $($see Section~\ref{subsec:tP}$)$, 
\item[-] at each limit ordinal step $0<\alpha<\theta$ it holds that $\widehat f_\alpha\in\widehat{\mathcal I}_{\mathbf s}$ and $g_\alpha=\mathbf s(\widehat f_\alpha)$, where $\theta$ is the length of the expansion  $\widehat f$ of $f$.
\end{itemize}
That is, the choice of germs $g_\alpha\in\mathcal G_{AN}$ at limit ordinal steps of the transfinite Poincar\' e algorithm is uniquely determined by the section $\mathbf s$. 
\end{defi}
%That is, a fixed section $\mathbf s$ attributes to every asymptotic transseries $\widehat g\in\widehat{\mathcal I}_{\mathbf s}\subseteq \widehat{\mathcal S}$ a  \emph{unique} germ $f\in\mathcal G$ whose asymptotic expansion with respect to section $\mathbf s$ is $\widehat f$. 
%\edz{Comment for Furthermore, ... It is not equivalent to saying that $f\in\mathcal G_{AN}$ admits $\widehat f$ as its $\mathbf s$-sectional asy. expansion since the section is not supposed to uniquely determine the germ at the final ordinal, only on limit ordinals strictly less than final. We only want uniqueness of the asy. exp. for the Fatou coordinate, that is,  injectivity expansion->germ, not germ->expansion.}
%\begin{obs}[Sectional asymptotic expansions and transfinite Poincar\' e algorithm]
%Let $\mathbf s$ be a fixed section and let $f\in \mathcal G_{AN}$ admit $\widehat f\in\widehat{\mathcal{I}}_{\mathbf s}$ as its $\mathbf s$-sectional asymptotic expansion. Then, using the notations of Section \ref{subsec:tP}, there exists a sequence $\widehat{f}_{\alpha}$ such that:

%1. if $\alpha=\nu+1$ is a \emph{successor ordinal}, then $\widehat{f}_{\alpha}$ is obtained from $\widehat{f}_{\nu}$ by the usual definition of Poincar\'e;

%2. if $\alpha$ is a \emph{limit ordinal}, then the choice of a germ $g_\alpha\in\mathcal G$ is uniquely dictated by the chosen section $\mathbf s$: $g_\alpha=s(\widehat{f}_{\alpha})$.
%\end{obs}

Notice that there exist germs from $\mathcal G$ that admit sectional asymptotic expansions in $\widehat{\mathfrak{L}}$ with respect to some section, but which do not admit sectional asymptotic expansion in $\widehat{\mathfrak{L}}$ with respect to some other section. Take, for example, the germ
$$
f(x)=x\cdot\Big(\frac{1}{1-\boldsymbol\ell}+x\cdot \sin\frac{1}{\boldsymbol\ell}\Big)
$$
and any section $\mathbf s$ such that $\mathbf s\big(x\sum_{k=0}^{\infty} \boldsymbol\ell^k\big)=\frac{x}{1-\boldsymbol\ell}$. Obviously, $f$ does not admit asymptotic expansion in $\widehat{\mathfrak L}$ with respect to $\mathbf s$. On the contrary, take any section $\mathbf s$ such that $\mathbf s\big(x\sum_{k=0}^{\infty}\boldsymbol\ell^k\big)=\frac{x}{1-\boldsymbol\ell}+e^{-\frac{1}{\boldsymbol\ell}}\sin\frac{1}{\boldsymbol\ell}$. Then the sectional asymptotic expansion of $f$ with respect to the section $\mathbf s$ is then equal to $\widehat f(x)=x\sum_{k=0}^{\infty}\boldsymbol\ell^k.$

\begin{prop} [Uniqueness of the $\mathbf s$-sectional asymptotic expansion] The $\mathbf s$-sectional asymptotic expansion of a germ $f\in\mathcal G$, if it exists, is unique.
%	 In particular, it also means that the order type of the expansion $\theta$ is uniquely determined by the section $\mathbf s$.\edz{I am right, am I? MAybe it is worth telling.}
	 \end{prop}
\noindent The proof follows by the definition of $\mathbf s$-sectional asymptotic expansions and the transfinite Poincar\' e algorithm from Subsection~\ref{subsec:tP}.
\smallskip

On the other hand, note that by Definition~\ref{def:sectional}, the injectivity of the mapping $f\mapsto\widehat f$, where $\widehat f$ is the $\mathbf s$-sectional asymptotic expansion of $f$, is \emph{not} implied. \\
%\edz{IMPORTANT!!!!}

It is clear that all the sectional asymptotic expansions of a germ $f\in\mathcal{S}$ have the same leading term (see Definition \ref{def:eljot}). Hence we can put:

\begin{defi} Let $f\in\mathcal S.$ The \emph{leading term of $f$}, denoted by 
	$\mathrm{Lt}(f)$, is the leading term $\mathrm{Lt}(\widehat f)$ of its any sectional asymptotic expansion $\widehat f\in\widehat{\mathcal S}$.  %In particular, $f\in \mathcal T$ is parabolic if $\text{Lt}(\widehat f)=(1,0,\ldots,0)$. 
%We say that a Dulac germ $f$ is parabolic if the order of its Dulac expansion $\widehat f\in\widehat{\mathcal L}$ is $\text{ord}(\widehat f)=(1,0)$.
\end{defi}

We define here what we mean by \emph{convergent transseries} in $\widehat{\mathfrak L}$, which are canonically summable in some sense. We introduce a desired property of sections: we call it \emph{coherence}. It means that a section respects convergence, that is, to a convergent transseries in $\widehat{\mathfrak L}$ assigns its sum.

\begin{defi}[Convergent transseries in $\widehat{\mathfrak L}$] \label{def:conv}Let $\widehat f\in\widehat {\mathcal L}_j^\infty$, $j\in\mathbb N_0$. We will say that $\widehat f$ given by \eqref{summable} is a \emph{convergent transseries} if \eqref{summable} is a  summable family of monomials $($summable pointwise on some open interval $(0,d)$, $d>0$$)$. That is, if there exists $d>0$ such that the multiple sum converges absolutely on $(0,d)$:
\begin{equation}\label{red}\sum_{i_{0}=0}^{\infty}\sum_{i_{1}=0}^{\infty}\cdots\sum_{i_{j}=0}^{\infty}\left|a_{i_{0}\ldots i_{j}}\right|x^{\alpha_{i_{0}}}\boldsymbol{\ell}^{\alpha_{i_{0}i_{1}}}\cdots\boldsymbol{\ell}_{j}^{\alpha_{i_{0}\cdots i_{j}}}<\infty,\quad x\in\left(0,d\right)
	.\end{equation}
In that case, by $f\in\mathcal G$ we denote the sum of $\widehat f$ on $(0,d)$ in the sense of summable families:
\begin{equation}\label{eq:summe}
\begin{alignedat}{1}f(x):=\sum_{i_{0}=0}^{\infty}\sum_{i_{1}=0}^{\infty} & \cdots\sum_{i_{j}=0}^{\infty}a_{i_{0}\ldots i_{j}}^{+}x^{\alpha_{i_{0}}}\boldsymbol{\ell}^{\alpha_{i_{0}i_{1}}}\cdots\boldsymbol{\ell}_{j}^{\alpha_{i_{0}\cdots i_{j}}}\\
& -\sum_{i_{0}=0}^{\infty}\sum_{i_{1}=0}^{\infty}\cdots\sum_{i_{j}=0}^{\infty}a_{i_{0}\ldots i_{j}}^{-}x^{\alpha_{i_{0}}}\boldsymbol{\ell}^{\alpha_{i_{0}i_{1}}}\cdots\boldsymbol{\ell}_{j}^{\alpha_{i_{0}\cdots i_{j}}},\quad x\in\left(0,d\right)
\end{alignedat}
\end{equation}

where $a_{i_0\ldots i_j}^+:=\max\{a_{i_0\ldots i_j},0\}>0$, $a_{i_0\ldots i_j}^-:=\max\{-a_{i_0\ldots i_j},0\}>0$.
\end{defi}
\noindent For exact definition and properties of summable families, see e.g. \cite{dieu}. Note that, due to positivity, the order of summation in \eqref{red} and \eqref{eq:summe} is not important.
 
\begin{defi}[Coherent sections]\label{remi} We say that a section $\mathbf{s}$ is \emph{coherent} if it respects \emph{convergence}. That is, if for every convergent $\widehat f\in\widehat{\mathcal I}_s\subseteq \widehat {\mathcal S}$ it holds that $\mathbf{s}(\widehat f)=f$, where $f$ is the sum of $\widehat f$ in the sense of Definition~\ref{def:conv}. 
\end{defi}
Let $\mathbf s$ be a coherent section. Let  $\widehat f$ convergent, with sum $f\in\mathcal G$, belong to $\widehat{\mathcal I}_\mathbf s$. By Fubini's theorem, all partial expansions of $f$, $\widehat f_\alpha$, with respect to $\mathbf s$, are also convergent. By Definition~\ref{def:sectional} of sections, it implies that also $\widehat f_\alpha\in \widehat{\mathcal I}_{\mathbf s}$. Therefore, the above definition of coherence is consistent.
%\edz{Changed and added.}

\subsection{Integral sections}\label{sec:is}

We define here a particular family of coherent sections which we call the \emph{integral sections}. The definition of integral section is based on Definition~\ref{defi}, which provides a formula for assigning a unique sum to a special type of divergent series in $\widehat{\mathcal L}_0^\infty$, which we call \emph{integrally summable}. The formula is adapted to the Fatou coordinate constructed in Section~\ref{sec:Fatou}.  It uses the fact that the Fatou coordinate, as a formal transseries or as a germ at $0$, is a solution of the Abel equation. See Example~\ref{ex:expa} (2) in Section~\ref{sec:examples} or Example~\ref{ex:exampleconstruction} in Section~\ref{sec:Fatou} for a better understanding.

%We will prove later in Section~\ref{sec:Fatou} that all sectional asymptotic expansions of a Fatou coordinate $\Psi\in\mathcal G_{AN}$ of a Dulac germ with respect to integral sections 
%differ only by an additive constant. Moreover, they also differ from the formal Fatou coordinate $\widehat\Psi\in\widehat{\mathcal L}_2^\infty$ by an additive constant.
%\smallskip

\begin{defi}[Integrally summable series in $\widehat {\mathcal L}_0^\infty$]\label{defi}\

$(1)$ By $\widehat{\mathcal L}_0^I\subset \widehat{\mathcal L}_0^\infty$ we denote the set of all formal series $$\widehat f(y)=\sum_{n=N}^{\infty}a_n y^n\in\widehat{\mathcal L}_0^\infty,\ N\in\mathbb Z,\ a_n\in\mathbb R,$$ 
which are either: 
\begin{enumerate}
\item[(i)] convergent on an open interval $(0,\delta)$, $\delta>0$, or 

\item[(ii)] divergent and such that
 there exists $\alpha\in\mathbb R$, $\alpha\neq 0$, for which 
\begin{equation}\label{deff}
\frac{d}{dx}\Big(x^ {\alpha} \widehat f(\boldsymbol\ell)\Big)=x^{\alpha-1} R(\boldsymbol\ell),
\end{equation}
 formally in $\widehat{\mathcal L}^\infty$,
where $R$ is a convergent Laurent series.

\end{enumerate}

\noindent We call $\widehat {\mathcal L}_0^I$ the set of \emph{integrally summable series of $\widehat{\mathcal L}_0^\infty$}.
\medskip

$(2)$ \begin{enumerate}
\item[(i)] If $\widehat{f}$ is convergent, we define its \emph{integral sum} as the usual sum $($on an interval $(0,\delta))$.

\item[(ii)] If $\widehat f\in\widehat {\mathcal L}_0^I$ is divergent, we define its \emph{integral sum} $f\in\mathcal G_{AN}$ by: 
\begin{equation}\label{eq:joj}
f(y):=
\frac{\int_d^{e^{-1/y}} s^{\alpha-1}R\big(\boldsymbol\ell(s)\big)\,ds}{e^{-\frac{\alpha}{y}}},
%\frac{-\int_{e^{-1/y}}^d g(t)\,d(e^{-1/t})}{e^{-\frac{\alpha}{y}}},&\alpha<0,\ d>0,
\end{equation}
where $\boldsymbol\ell(s)=-\frac{1}{\log s}$. We put $d=0$ if $s^{\alpha-1}R\big(\boldsymbol\ell(s)\big)$ is integrable at $0$ $($i.e. $\alpha>0)$ and $d>0$ otherwise $($i.e. $\alpha<0)$.
\end{enumerate}
\end{defi} 

\begin{obs} Note that the two parts of the definition of the integral sum 
are consistent. Indeed, in the case when $\widehat{f}$ is convergent, let $R$ be a convergent series defined by \eqref{deff} for $\alpha=0$. Now, the integral expression  \eqref{eq:joj} for $f(y)$, with $\alpha=0$, $d>0$, coincides with the sum of $\widehat{f}$ up to a constant. 
\end{obs}

\begin{obs} 

The idea behind the definition of an integral sum $f$ of $\widehat f$ is that $f$ is defined as a germ from $\mathcal G_{AN}$ such that the same equation as \eqref{deff} is satisfied, but in the sense of germs:
$$
\frac{d}{dx}\big(x^\alpha f(\boldsymbol\ell)\big)=x^{\alpha-1}R(\boldsymbol\ell),\ x\in(0,d),\ d>0.
$$
 Note that the solution of \eqref{deff} in the sense of germs is given by \eqref{eq:joj}.
\end{obs}
\smallskip

\begin{prop}\label{alph} The exponent $\alpha\neq 0$ in Definition~\ref{defi} $(1)\,(ii)$ is unique. We call such $\alpha$ the \emph{exponent of integration of $\widehat f$}.
\end{prop} 
\noindent The proof is in the Appendix. 
Note that if \eqref{deff} were true for $\alpha=0$, it would imply that $\widehat f$ is convergent, so we can suppose $\alpha\neq 0$ for divergent series. 
\medskip

\begin{obs}\label{rem:uniqLo}\

(1) Note that the integral sum $f$ of $\widehat f\in\widehat{\mathcal L}_0^I$, as defined by \eqref{eq:joj}, is 
unique for $\alpha>0$. It is 
\emph{unique only up to $Ce^{\frac{\alpha}{y}}$, $C\in\mathbb R$}, for $\alpha<0$, due to the possible choice of $d>0$. For $\alpha=0$ (that is, if $\widehat f$ is convergent), the integral sum is unique.
\smallskip

(2) Proposition~\ref{prop:rational_fct} in the Appendix shows that $\widehat f\in\widehat{\mathcal L}_0^I$ is the power asymptotic expansion of its integral sums.
In the case $\alpha<0$, the term $Ce^{\frac{\alpha}{x}}$ in the integral sum $f$ is exponentially small with respect to monomials in $\widehat{\mathcal L}_0$. Therefore, adding such a term to $f$ does not change its asymptotic expansion $\widehat f\in\widehat{\mathcal L}_0$.
%\edz{Reformulated.}

 %The series $\widehat f\in\widehat{\mathcal L}_0^I$ is the power asymptotic expansion up to a constant $C\in\mathbb{R}$ of its integral sums, for $\alpha=0$.
\end{obs}
\smallskip

\begin{defi}[Integral sections]\label{def:is} The vector subspace $\widehat{\mathcal I}$ of $\widehat{\mathcal S}$ generated by the transseries $\widehat{g}(x)=x^{\alpha}\widehat f(\boldsymbol\ell)$, where $\widehat f\in\widehat {\mathcal L}_0^I$ and $\alpha\in\mathbb{R}$ is the exponent of integration of $\widehat f$ (or an arbitrary real number if $\widehat f$ convergent), is called the space of \emph{integrally asymptotic transseries}.

An \emph{integral section} is any coherent section $\mathbf s:\widehat{\mathcal I}\to \mathcal G_{AN}$ which, to every generator $\widehat{g}(x)=x^{\alpha}\widehat f(\boldsymbol\ell)$ of $\widehat{\mathcal{I}}$, associates $\mathbf{s}(\widehat{g})=x^{\alpha}f(\boldsymbol{\ell})$, where 
$f\in\mathcal{G}_{AN}$ 
%	\widehat f\in\widehat {\mathcal L}_0^I$ and $\alpha\in\mathbb{R}$  ^{\alpha_i}\widehat f_i(\boldsymbol\ell),$$ attributes the germ $f=\mathbf{s}(\widehat f)\in\mathcal G_{AN}$ given by \begin{equation}\label{eq:g}f(x)=\sum_{i=1}^{n}x^{\alpha_i} f_i(\boldsymbol\ell).\end{equation} Here, $f_i\in\mathcal G_{AN}$ are 
is the integral sum of $\widehat f$ given by \eqref{eq:joj}. 
\end{defi}
Note that the germ $f$ is unique only up to an additive constant (Remark~\ref{rem:uniqLo}). Likewise, the integral sections $\mathbf s$ are \emph{linear} mappings \emph{only up to an additive constant}.
%Trivially, $\widehat{\mathcal L}_0^I\subset\widehat {\mathcal S}_0$. The convergent series in $\widehat {\mathcal S}_0$  belong to $\widehat{\mathcal L}_0^I$ by its definition. Note that for $\widehat f_i(\boldsymbol\ell)\in\widehat{\mathcal L}_0^I$ with exponents of integration $\alpha_i\in\mathbb R$ it necessarily holds that $\sum_{i=1}^{n} x^{\alpha_i}\widehat f_i(\boldsymbol\ell)\in\widehat{\mathcal S}_1\subset\widehat{\mathcal S}$. Indeed, by Remark~\ref{rem:uniqLo}, the sum $\sum_{i=1}^{n} x^{\alpha_i} f_i(\boldsymbol\ell)\in\mathcal G_{AN}$, where $f_i$ is an integral sum of $\widehat f_i$, admits an asymptotic expansion $\sum_{i=1}^{n} x^{\alpha_i}\widehat f_i(\boldsymbol\ell).$ Note also that on the set of transseries of the form $\sum_{i=1}^{n} x^{\alpha_i}\widehat f_i(\boldsymbol\ell)\in\widehat{\mathcal S}_1$, where $\widehat f_i\in\widehat {\mathcal L}_0^I\subset\widehat{\mathcal L}_0^\infty$, $\alpha_i\in\mathbb R$ the exponents of integration of $\widehat f_i$ (or arbitrary real numbers if $\widehat f_i$ convergent) and $n\in\mathbb N$, the formula \eqref{eq:g} is \emph{linear}. Therefore, the integral sections are well-defined.  
\smallskip 

\section{Embedding in a one-parameter flow}

\label{sec:embedding} In this section (Propositions~\ref{def:formfatou}
and \ref{prop:exi}) we discuss the close relationship between the existence
of a (formal) Fatou coordinate for a germ $f\in\mathcal G_{AN}$ (resp.
$\widehat{f}$) and an embedding of $f\in\mathcal G_{AN}$ (resp. 
$\widehat{f}$) in a (formal) one-parameter flow.  For the explicit
relation, consult also the constructive proofs of propositions in
the Appendix.

%The existence and the uniqueness of a Fatou coordinate for $f$ (\emph{resp}.
%$\widehat{f}$) is closely related to the embedding of $f$ (\emph{resp}.
%$\widehat{f}$) in a \emph{one-parameter flow}. 

\begin{defi}[One-parameter local flow, standard definition] Let $a$
%\edz{Changed to local definition. Should be properly checked (properties)}
 and $b$ be two continuous functions defined on $(0,d)$ such that $a(x)<0<b(x)$. A family $\{f^{t}\}_{t\in (a(x),b(x))}$
of functions defined on some open interval $(0,d)$, $d>0$, is called a\emph{ one-parameter
flow on $(0,d)$} if: 
\begin{enumerate}
\item $f^{0}(x)=x$, $x\in (0,d)$,
\item If $s\in (a(x),b(x))$ and $t\in (a(f^{s}(x)),b(f^{s}(x)))$ then $t+s\in (a(x),b(x))$ and $f^{t+s}(x)=f^{t}(f^{s}(x)),\ x\in(0,d).$ 
\end{enumerate}
\noindent We say that the one-parameter flow $\{f^{t}\}$ is a \emph{$C^{1}$-flow}
if the mapping $t\mapsto f^{t}(x)$ is of class $C^{1}\big((a(x),b(x))\big)$,
for every $x\in(0,d)$. \smallskip{}

Let $f\in \mathcal G_{AN}$.
%\edz{??? Is it ok?}
 We say that
$f$ \emph{embeds as the time-one map in a flow $\{f^{t}\}$} if $f^{1}=f$ $($that is, if $f^1$ exists and is equal to $f$ on some interval $(0,d))$.
\end{defi}

\begin{defi}[Formal one-parameter flow, see \cite{mrrz2}] We say
that a one-parameter family $\{\widehat{f}^{t}\}_{t\in\mathbb{R}}$,
$\widehat{f}^{t}\in\widehat{\mathcal{L}}$, is a \emph{$C^{1}$-formal
flow} if: 
\begin{enumerate}
\item $\widehat{f}^{0}=\mathrm{id},\ \widehat{f}^{t+s}=\widehat{f}^{t}\circ\widehat{f}^{s}$,\ \ $s,\, t\in\mathbb{R}$, 
\item $S:=\cup_{t\in\mathbb{R}}\mathrm{Supp}(\widehat{f}^{t})$ is a well-ordered
subset of $\mathbb{R}_{>0}\times\mathbb{Z}$, 
\item $t\mapsto[\widehat f^{t}]_{(\alpha,m)}$ is of the class $C^{1}(\mathbb{R})$,
for every $(\alpha,m)\in S$. 
\end{enumerate}
Here, $[\widehat f^{t}]_{(\alpha,m)}$ denotes the coefficient (which is a
function of $t$) of the monomial $x^{\alpha}\boldsymbol{\ell}^{m}$
in $\widehat{f}^{t}$. Again, we say that $\widehat{f}\in\widehat{\mathcal{L}}$
embeds in the flow $\{\widehat{f}^{t}\}_{t\in\mathbb{R}}$ as the
time-one map if $\widehat{f}=\widehat{f}^{1}$. \end{defi}

In \cite{mrrz2}, we have proved that any \emph{parabolic} transseries
$\widehat{f}\in\widehat{\mathcal{L}}$ is the time-one
map of a unique $C^{1}$-flow in $\widehat{\mathcal{L}}$. The next
proposition states that, accordingly, the formal Fatou coordinate
exists and is unique. 

\begin{prop}[Existence and uniqueness of the formal Fatou coordinate]\label{def:formfatou}
Let $\widehat{f}\in\widehat{\mathcal{L}}$ be parabolic. The formal Fatou
coordinate $\widehat{\Psi}$
exists and is unique in $\widehat{\mathfrak{L}}$, up
to an additive constant. Moreover, $\widehat\Psi\in\widehat{\mathcal L}_2^\infty$. Let $\{\widehat{f}^{t}\}$ be the (unique)
$C^{1}$-flow in $\widehat{\mathcal{L}}$ in which $\widehat{f}$
embeds as the time-one map. Then: 
\begin{equation}
\widehat{\Psi}(\widehat{f}^{t}(x))-\widehat{\Psi}(x)=t,\ t\in\mathbb{R}.\label{fatou}
\end{equation}
\end{prop}

For a more precise description of the formal Fatou coordinate
of a parabolic $\widehat{f}\in\widehat{\mathcal{L}}$, see Remark~\ref{rem:class}
in the Appendix. If $\widehat f$ is moreover parabolic and Dulac, the description is provided in Proposition~\ref{prop:exformfatou}.\\

The next proposition establishes, for an analytic germ in $\mathcal G_{AN}$, the equivalence
between the existence of an analytic Fatou coordinate and the embedding
of the germ in a $C^1$-flow:

\begin{prop}\label{prop:exi}
	%\edz{Check the proposition. I added the attracting condition and reformulated a little-we need to be more precise about the interval than just \emph{germs}. Do you agree?}
	Let $f$ be an analytic function defined on some open interval $(0,d)$, $d>0$, such that $f(0)=0$ (in the limit sense) and $0<f(x)<x$, $x\in(0,d)$.\footnote{That is, $0$ is the only fixed point in $(0,d)$. Moreover, it is attracting.}
The following statements are equivalent:

1. there exists a Fatou coordinate
$\Psi$ for $f$ defined and analytic on $(0,d)$,

2. there exists a $C^{1}$-flow $\{f^{t}\}$, $f_t$ defined and analytic on
$(0,d)$, in which $f$ can be embedded as the time one-map, and such
that the function $\xi$ defined on $(0,d)$ by 
\[
\xi:=\frac{d}{dt}f^{t}\Big|_{t=0}
\]
is non-oscillatory\footnote{non-oscillatority means that there is no accumulation of zero points at $0$}. 

In this case, for the Fatou coordinate $\Psi$ and for the corresponding
flow $\{f^{t}\}$, the following holds: 
\begin{equation}
%\edz{Simply added $t\in(\mathbb R,0)$, not very precise here.}
\Psi(f^{t}(x))-\Psi(x)=t,\ x\in(0,d),\ t\in(\mathbb{R}_+,0).\label{eq:fat2}
\end{equation}
\end{prop}

\noindent The proofs of Propositions~\ref{def:formfatou} and \ref{prop:exi},
are given in the Appendix. Also, see Remark~\ref{rem:oscil} in the
Appendix for the explanation of the importance of the \emph{non-oscillatority
assumption} in Proposition~\ref{prop:exi}.\\

In the particular case of a \emph{parabolic Dulac germ} $f$, the
existence, uniqueness and the description of the Fatou coordinate for
$f$, as well as of the formal Fatou coordinate for its Dulac expansion $\widehat{f}$,
are given in Section~\ref{sec:Fatou} (in the proof of the Theorem) by means of an explicit construction.

\section{Examples}

\label{sec:examples}
\begin{example}[Examples of sectional asymptotic expansions in $\widehat{\mathfrak L}$]\label{ex:expa}~

(1) The asymptotic expansion of a Dulac germ $f\in\mathcal G_{AN}$ is unique and equal to its Dulac expansion:
\[
\widehat{f}=\sum_{i=1}^{\infty}x^{\alpha_{i}}P_{i}(\boldsymbol{\ell})\in\widehat{\mathcal L},
\]
where $(\alpha_{i})_{i}$ is an increasing sequence of finite type (Definition~\ref{def:fgs}) and
$P_{i}$ is a sequence of polynomials. Indeed, since $P_i$ are polynomials (finite sums of powers of $\boldsymbol\ell$), the asymptotic expansion $\widehat f$ of a Dulac germ $f$ is unambigously defined by the standard Poincar\' e algorithm.
\medskip

(2)  	The time-one map $f\in\mathcal G_{AN}$ of the vector field $X=\xi(x)\frac{d}{dx},\ \xi(x)=x^2\boldsymbol\ell(x)^{-1},$ admits a unique Poincar\' e asymptotic expansion $\widehat f\in\widehat{\mathcal L}$ given by the formal exponential $\widehat f(x)=\mathrm{Exp}\big(x^2\boldsymbol\ell^{-1}\frac{d}{dx}\big).\mathrm{id}$. It is not transfinite, since every power of $x$ in the expansion is multiplied by only finitely many powers of $\boldsymbol\ell$.

A Fatou coordinate $\Psi\in\mathcal G_{AN}$ of $f$ is computed by 
\begin{equation}\label{ha}
\Psi(x)=\int_d^x \frac{1}{\xi(x)}\,dx=\int_d^x x^{-2}\boldsymbol\ell(x)\,dx, \ d>0.\nonumber
\end{equation}
On the other hand, a formal Fatou coordinate of its expansion $\widehat f$ can be computed as a formal antiderivative in $\widehat{\mathfrak L}$ of $x^{-2}\boldsymbol\ell$ without constant term\footnote{By \emph{formal antiderivative in $\widehat{\mathfrak L}$} of a transseries $\widehat f\in\widehat{\mathfrak L}$ \emph{without constant term} we mean a transseries from $\widehat{\mathfrak L}$ obtained by integrating $\widehat f$ monomial by monomial, with constant of integration equal to $0$.}: $$\widehat\Psi(x)=\int  \frac{dx}{\xi(x)}=\int x^{-2}\boldsymbol\ell\,dx,$$ see the proof of Proposition~\ref{def:formfatou}. By integration by parts,
$$
\widehat\Psi(x)=-x^{-1}\sum_{n=1}^{\infty} n! \boldsymbol\ell^{n}.
$$
The above series $\widehat T(\boldsymbol\ell):=\sum_{n=1}^{\infty} n! \boldsymbol\ell^{n}$ is a divergent series in $\widehat{\mathcal L}_0$. There are many ways to find a germ $T\in\mathcal G_{AN}$ which admits $\widehat T$ as its power asymptotic expansion. We choose one particular sum, the \emph{integral sum} from Definition~\ref{defi}, which is uniquely defined up to a term $Cx$. That is, we choose an \emph{integral} section $\mathbf s$. By Definition~\ref{def:is},
$$
T:=\frac{\int_d^x x^{-2}\boldsymbol\ell(x)\,dx}{-x^{-1}}\in\mathcal G_{AN},\ d>0.
$$
Due to the choice of $d>0$, we get integral sums differing by a constant.
Up to a constant, $\Psi$ is the sectional sum of $\widehat\Psi$ with respect to the integral section $\mathbf s$ and $\widehat\Psi$ is the Poincar\' e asymptotic expansion of $\Psi(x)=-x^{-1}T(x)\in\mathcal G_{AN}$.
\medskip

(3) Let $\mathbf s$ be a coherent section as defined in Definition~\ref{remi}. The following examples of germs in $\mathcal G_{AN}$ admit sectional asymptotic expansions with respect to $\mathbf s$ in  $\widehat {\mathfrak L}$:
\[
f(x):=F(x,\boldsymbol{\ell},\boldsymbol{\ell}_{2}),\  g(x):=G(x,\frac{x}{\boldsymbol\ell}),
\]
where $F$ is a germ at the origin of a function of three variables analytic at $(0,0,0)$ and $G$ a germ at the origin of a function of two variables analytic at $(0,0)$.

Let us show the statement for $f$ (for $g$ analogously). For $y,z$ small enough, the germ $x\mapsto F(x,y,z)$ is analytic at $0$. Therefore, for every $n\in\mathbb{N}$, we have:
\[
F(x,y,z)=\sum_{k=0}^{n}g_{k}(y,z)x^{k}+R_{n}(x,y,z),
\]
where $x\mapsto R_{n}(x,y,z)$ is analytic, for \emph{fixed} small
$y,z$. Take the \emph{Lagrange form of the remainder} $R_{n}(x,y,z)$,
for $y,\, z$ fixed: 
\[
R_{n}(x,y,z)=\int_{0}^{x}\partial_{t}^{n+1}F(t,y,z)\frac{(x-t)^{n}}{n!}dt
\]
Since $F$ is analytic in $y,z$, we get that $R_{n}$ is analytic
in the three variables, and, moreover, we get the following \emph{uniform}
bound in $y,\, z$: 
\[
|R_{n}(x,y,z)|\leq C\frac{x^{n+1}}{(n+1)!},
\]
on some small ball around $(0,0,0)$. Since $R_{n}$ is analytic in
$y,\, z$, $n\in\mathbb{N}$, we conclude that the functions $g_{k}(y,z)$
are also analytic in $y,\, z$, for all $k=0,\ldots,n$, and we
can repeat the same procedure for their expansions at the next level.
Finally, $f$ admits the unique sectional asymptotic expansion with respect to $\mathbf s$ equal to 
$\widehat F(x,\boldsymbol\ell,\boldsymbol\ell_2)\in \widehat{\mathcal L}_2 ,
$
where $\widehat F$ is the Taylor expansion of $F$ at $(0,0,0)$.
\medskip

As a simple example of this type of germ, take e.g. 
$$
f(x)=\frac{1}{1-\frac{x}{1-\frac{\boldsymbol{\ell}}{1-\boldsymbol{\ell}_{2}}}}.
$$
Its sectional asymptotic
expansion in $\widehat{\mathcal{L}}_{2}$ with respect to $\mathbf s$ is given by the transseries: 
$$
\widehat{f}=\sum_{k=0}^{\infty}x^{k}\Big(1-\frac{\boldsymbol{\ell}}{1-\boldsymbol{\ell}_{2}}\Big)^{-k}=\sum_{k=0}^{\infty}\sum_{l=0}^{\infty}\sum_{r=0}^{\infty}{-k \choose l}{-l \choose r}(-1)^{l+r}x^{k}\boldsymbol{\ell}^{l}\boldsymbol{\ell}_{2}^{r}.
$$
\end{example}

\section{Proof of the Theorem}\label{sec:Fatou}

The proof of the Theorem is constructive. 
Before proving the Theorem, we give an example illustrating our construction of the Fatou coordinate.
\begin{example}\label{ex:exampleconstruction} Take a Dulac germ $f\in\mathcal G_{AN}$ with the expansion:
$$
\widehat f(x)=x-x^2\boldsymbol\ell^{-1}+o(x^3)
$$The algorithm which will be described in this section is a block-by-block construction of the formal Fatou coordinate $\widehat \Psi\in\widehat{\mathfrak L}$ satisfying formally in $\widehat{\mathfrak L}$ the Fatou equation:
$$
\widehat \Psi \big(x-x^2\boldsymbol\ell^{-1}+o(x^3)\big)-\widehat\Psi(x)=1.
$$
\emph{Blocks} of a transseries from $\widehat{\mathfrak L}$ are defined in Definition~\ref{def:eljot}.
%\edz{Added reference to definition of blocks.}
Let $\widehat\Psi_1$ be the leading block of $\widehat\Psi$, that is, $\widehat\Psi=\widehat\Psi_1+h.o.b.$ Here, $h.o.b.$ stands for \emph{blocks of strictly higher order $($in $x)$}. Applying the formal Taylor expansion, the lowest-order block on the left-hand side is $-\widehat\Psi_1'(x)\cdot x^2 \boldsymbol\ell^{-1}$, so it should equal $1$ on the right-hand side. Therefore, $\widehat\Psi_1$ is given as the 
%\edz{Added.}
formal antiderivative of $-x^2\boldsymbol\ell$ in $\widehat {\mathfrak L}$ without constant term (see footnote on p. \pageref{ha}):
\begin{equation}\label{eq:pssi}
\widehat\Psi_1(x)=\int\frac{dx}{x^{2}\log x}.
\end{equation}
To continue, put $\widehat\Psi=\widehat\Psi_1+\widehat R$ in the Fatou equation and repeat the procedure for the following blocks.
By formal integration by parts, we see that 
\begin{equation}\label{eq:exfatou}
\widehat{\Psi}_1(x)=x^{-1}\sum_{n=1}^{\infty}n!\boldsymbol{\ell}^{n}.
\end{equation}
For every $\boldsymbol{\ell}\in(0,d)$, the numeric series $\sum_{n=1}^{\infty}n!\boldsymbol{\ell}^{n}$
%\edz{I am not sure to understand a comment, I did my best.}
is divergent (and the formal series $\widehat f(\boldsymbol\ell)=\sum_{n=1}^{\infty}n!\boldsymbol{\ell}^{n}$ is Borel summable). However, following our construction, this series
is \emph{uniquely real summable}. That is, we also perform a parallel block by block construction on the level of germs. Let 
\[
\Psi_1(x):=\int_{d}^{x}\frac{dt}{t^{2}\log t},\ d>0,
\]
be an analytic analogue of $\widehat\Psi_1$ from \eqref{eq:pssi}. Note that different choices of $d>0$ lead to germs $\Psi_1$ differing by a constant.
Obviously, $\Psi_1\in\mathcal G_{AN}$ and therefore the integral sum  $f(\boldsymbol\ell):=e^{-1/\boldsymbol{\ell}}\Psi_1(e^{-1/\boldsymbol{\ell}})$   of $\widehat{f}$ belongs to $\mathcal G_{AN}$. 
By Proposition~\ref{prop:rational_fct}, the power asymptotic expansion of $f\in\mathcal G_{AN}$ is equal to $\widehat{f}$. This procedure is formalized in Definition~\ref{defi}. Consequently, the Poincar\' e asymptotic expansions of $\Psi_1$ are, up to additive constants, equal to the formal series $\widehat\Psi_1$.
\end{example}
\medskip

\subsection{The proof of the Theorem}\

\emph{} Point $1.$ has already been proven in Proposition~\ref{def:formfatou}.
We now give a constructive proof of the existence of $\widehat\Psi\in\widehat{\mathcal L}_2^\infty$ from 1. As illustrated in Example~\ref{ex:exampleconstruction}, simultaneously with formal Fatou coordinate $\widehat\Psi$, we construct a Fatou coordinate $\Psi\in\mathcal G_{AN}$ from point $2.$ and prove the relation between $\widehat\Psi$ and $\Psi$ from $3.$ and $4$. We follow in large part the construction of the Fatou coordinate for parabolic
diffeomorphisms as explained, for example, in \cite{loray}.
\medskip

We proceed in four steps:

\emph{Step 1.} In Subsection~\ref{sec:subzero}, we construct the formal
Fatou coordinate $\widehat{\Psi}$, by solving \emph{block by block}
the formal Abel equation. By \emph{block}, we mean the (formal) sum
of all the monomials of $\widehat{f}$ which share a common power
of $x$.  We also get the precise form \eqref{eq:ffd} of $\widehat\Psi$. Recall that the formal Fatou coordinate is unique by Proposition~\ref{def:formfatou}.

Simultaneously, we provide the ``block by block'' construction of a
Fatou coordinate $\Psi\in\mathcal G_{AN}$, where the germ for each block (from $\mathcal G_{AN}$) is represented
by an integral. We prove that each formal block is the Poincar\' e asymptotic expansion of the corresponding integrally defined germ from $\mathcal G_{AN}$, at least up to a constant term (Remark~\ref{rem:uniqLo}).

%By construction, if the integral section function $\mathbf s$ is chosen appropriately to the choice of constant term in $\widehat\Psi$ (see the comment following  Proposition~\ref{alph}), $\Psi$ admits the formal coordinate $\widehat\Psi$ as its unique sectional asymptotic expansion with respect to $\mathbf s$. 

Additionally, we control the support of the formal Fatou coordinate, and prove that
the powers of $x$ it contains belong to a \emph{finitely generated\footnote{Every element is a linear combination over positive integers of finitely many generators of the lattice.}\label{haa}
lattice
%\edz{Added a footnote because the referee asked many times what we mean by finitely generated.}
}.

\emph{Step 2.} (see Subsection~\ref{sec:subone}). The control of the support
obtained in the previous step allows to conclude that the principal
part $\widehat{\Psi}_{\infty}$ of the Fatou coordinate is obtained
after finitely many steps of the ``block by block'' algorithm. We
prove also that the principal part $\widehat{\Psi}_{\infty}$ is the sectional asymptotic expansion  in
$\widehat{\mathfrak L}$ with respect to any integral section $\mathbf s$ of the  principal part $\Psi_{\infty}$, up to a constant term depending on the choice of the integral section. That is, depending on the choice of constants of integration in the integrals for the infinite part.

\emph{Step 3.} In Subsection~\ref{sec:subtwo}, we solve the modified
Abel equation \eqref{eq:modi} for the remaining infinitesimal part
of the Fatou coordinate $\widehat{R}$. The infinitesimal part $R$
is obtained directly from the equation in the form of a \emph{convergent}
series, following the method explained in \cite{loray}. We prove
in Section~\ref{sec:subfour} that the formal infinitesimal part $\widehat R$
of $\widehat{\Psi}$ obtained blockwise (in countably many steps) is indeed the sectional asymptotic expansion in $\widehat{\mathcal{L}}$ with respect to any integral section $\mathbf s$
of the infinitesimal part $R$ of $\Psi$.

\emph{Step 4.} Finally, in Subsection~\ref{sec:subthree}, we prove
the uniqueness of the formal Fatou coordinate $\widehat{\Psi}\in\widehat{\mathfrak{L}}$ up to an additive constant.
Furthermore, we prove the uniqueness, up to an additive constant, of the Fatou coordinate $\Psi\in \mathcal G_{AN}$
admitting an integral
%\edz{Added integral.}
sectional asymptotic expansion in $\widehat{\mathfrak L}$. Moreover, for a fixed constant term in $\widehat\Psi\in\widehat{\mathcal{L}}_{2}^{\infty}$ and for a fixed integral section $\mathbf s$, any such Fatou coordinate $\Psi\in \mathcal G_{AN}$ constructed in the proof admits $\widehat{\Psi}$ as its sectional asymptotic expansion with respect to $\mathbf s$, up to an additive constant.

If we change the integral section $\mathbf s$, the sectional asymptotic expansion with respect to $\mathbf s$ of a fixed\textsc{} Fatou coordinate $\Psi\in \mathcal G_{AN}$ changes only by an additive constant.

\subsubsection{\bf{The Fatou coordinate and the control of the support}.}

\label{sec:subzero}\

\medskip

Let 
\[
f(x)=x-x^{\alpha_{1}}P_{1}(\boldsymbol{\ell}^{-1})-x^{\alpha_{2}}P_{2}(\boldsymbol{\ell}^{-1})+o(x^{\alpha_{2}})
\]
be a parabolic Dulac germ. Here, $1<\alpha_{1}<\alpha_{2}<\ldots$
is an increasing sequence of finite typeor a strictly increasing sequence tending to $+\infty$, whose elements
belong to a finitely generated (see the footnote$^{5}$) sub-semigroup of $(\mathbb{R}_{>0},+)$
%$\edz{Removed $+$ and put it in brackets to mean semi-group by $+$-operation},
and $P_{i}$ are \emph{polynomials}. Let $\widehat{f}\in\widehat{\mathcal{L}}$
be its Dulac expansion. \\

We construct the formal Fatou coordinate $\widehat{\Psi}$ satisfying
the formal Abel equation 
\begin{equation}
\widehat{\Psi}(\widehat{f})-\widehat{\Psi}=1\label{eq:ab}
\end{equation}
\emph{block by block} and we control its support. By one \emph{block},
we mean the sum of all monomials sharing a common power of $x$, as defined in Definition~\ref{def:eljot}. We
call the power of $x$ of a block the \emph{order of the block.} In
each step, we consider the construction from two sides:

1. the side of formal transseries, and

2. the side of analytic germs in $\mathcal G_{AN}$.

\noindent Let us now describe the induction step. Put 
\[
\widehat{\Psi}=\widehat{\Psi}_{1}+\widehat{R}_{1},
\]
where $\widehat{\Psi}_{1}$ represents the lowest-order block of $\widehat\Psi$. Since we search for a solution $\widehat{\Psi}$ in
$\widehat{\mathfrak{L}}$, and since $\widehat{f}$ is parabolic,
we can expand the Abel equation \eqref{eq:ab} using Taylor expansions:
\[
\widehat{\Psi}'\cdot\widehat{g}+\frac{1}{2!}\widehat{\Psi}''\cdot\widehat{g}^{2}+\cdots=1,
\]
where $\widehat{g}=\mathrm{id}-\widehat{f}=x^{\alpha_{1}}P_{1}(\boldsymbol{\ell}^{-1})+x^{\alpha_{2}}P_{2}(\boldsymbol{\ell}^{-1})+\cdots$.
The term $(\widehat{\Psi}_{1})'\cdot x^{\alpha_{1}}P_{1}(\boldsymbol{\ell}^{-1})$
is the block of the strictly lowest order on the left-hand side, so
it should equal $1$ on the right-hand side. Therefore, 
\[
(\widehat{\Psi}_{1})'=\frac{x^{-\alpha_{1}}}{P_{1}(\boldsymbol{\ell}^{-1})}=x^{-\alpha_{1}}Q_{1}(\boldsymbol{\ell}).
\]
Here, $Q_{1}$ is a rational function. We obtain the formal antiderivative
$\widehat{\Psi}_{1}^{\infty}\in\widehat{\mathcal{L}}_{2}^{\infty}$ by expanding
the integral \emph{formally, using integration by parts}, without adding constant term, as explained in the footnote on p. \pageref{ha}
%\edz{Added}
(see Proposition~\ref{prop:rational_fct} in the Appendix).

On the other hand, we define: 
\[
\Psi_{1}(x):=
\int_{d}^{x}t^{-\alpha_{1}}Q_{1}(\boldsymbol{\ell})\, dt,\ d>0.
\]
Obviously, $\Psi_{1}\in\mathcal G_{AN}$. 
%\edz{Explanation added in a footnote.}
Note that $\alpha_1>1$ or $(\alpha_{1}=1,\ $\footnote{Here and in the sequel, by $\mathrm{ord}(Q_1)$ we mean the order (as defined after Definition~\ref{def:eljot}) of the power series $\widehat Q_1(\boldsymbol\ell)$ in $\boldsymbol\ell$, which is the power asymptotic expansion of the rational function $Q_1$ at $\boldsymbol\ell=0$.} $\text{ord}(Q_{1})\leq 1)$, so $\Psi_1(x)\to\infty,$ as $x\to 0$. Note that $\Psi_1$ is unique only up to an additive constant (free choice of $d>0$). In the sequel, this will be the case for infinite blocks (i.e. which tend to $\infty$, as $x\to 0$). On the other hand, infinitesimal blocks (which tend to $0$, as $x\to\infty$) will give unique germs.

By Proposition~\ref{prop:rational_fct} in the Appendix,
$\widehat{\Psi}_{1}$ is the sectional asymptotic expansion in $\widehat{\mathcal{L}}_{2}^{\infty}$
of $\Psi_{1}\in\mathcal{G}_{AN}$ with respect to any integral section, up to an additive constant.

The Abel equation for $R_{1}$ becomes: 
\begin{equation}
R_{1}(f(x))-R_{1}(x)=1-\big(\Psi_{1}(f(x))-\Psi_{1}(x)\big)=1-\int_{x}^{f(x)}t^{-\alpha_{1}}Q_{1}(\boldsymbol{\ell})\, dt.\label{eq:diff}
\end{equation}
Let us denote by $\delta_{1}(x):=1-\int_{x}^{f(x)}t^{-\alpha_{1}}Q_{1}(\boldsymbol{\ell})\, dt$
the new right-hand side of the equation. Obviously, $\delta_{1}\in\mathcal G_{AN}$
, as a difference of analytic germs.

On the other hand, applying Taylor expansion to $\widehat{\Psi}_{1}$, the \emph{formal} Abel equation becomes: 
\begin{align}\label{eq:r11}
\widehat{R}_{1}(\widehat{f}(x))-\widehat{R}_{1}(x) & =1-\big(\widehat{\Psi}_{1}(\widehat f(x))-\widehat{\Psi}_{1}(x)\big),\\
 & =1-(\widehat{\Psi}_{1})'\widehat{g}-\frac{1}{2!}(\widehat{\Psi}_{1})''\widehat{g}^{2}+\cdots,\nonumber\\
 & =1-\frac{x^{-\alpha_{1}}}{P_{1}(\boldsymbol{\ell}^{-1})}\widehat{g}-\frac{1}{2!}\Big(\frac{x^{-\alpha_{1}}}{P_{1}(\boldsymbol{\ell}^{-1})}\Big)'\widehat{g}^{2}+\cdots=:\widehat{\delta}_{1}(x).\nonumber
\end{align}
Note that, 
%\edz{Added. Check JP if the reference is right.}
due to the fact that the power of $x$ in the leading term of $\widehat f$ is strictly bigger than $1$, the Taylor formula above converges formally and is true in the formal setting. This is a classical result that can be checked in e.g. \cite{dries}. We denote the right-hand side of \eqref{eq:r11} by $\widehat\delta_1\in\widehat{\mathcal L}_1^\infty$. 
It can be checked from the above computation that the leading block
in $\widehat{\delta}_{1}$ is of order $\min\{\alpha_{1}-1,\alpha_{2}-\alpha_{1}\}$.
Similarly, we have $\delta_{1}=O(x^{\min\{\alpha_{2}-\alpha_{1},\alpha_{1}-1\}-\varepsilon})$,
for every $\varepsilon>0$. 
%\edz{Added about the expansion.}
By parallel constructions, $\widehat\delta_1$ is the sectional asymptotic expansion of $\delta_1$ with respect to a coherent section (series in $\boldsymbol\ell$ are convergent).\\

Furthermore, it can be seen that $\widehat{\delta}_{1}$ consists
of blocks of the type $x^{\beta}R(\boldsymbol{\ell}^{-1})$, where
$R$ is a rational function whose denominator can only be a \emph{positive}
integer power of the polynomial $P_{1}(\boldsymbol\ell^{-1})$, and $\beta$ belongs to
the set: 
\[
\mathcal{R}_{1}:=\{(\alpha_{n_{1}}+\cdots+\alpha_{n_{k}}-k)-(\alpha_{1}-1),\ k\in\mathbb{N}\},
\]
where $\alpha_{n_{i}}>1$ are powers of $x$ in $\widehat{f}$. Since
the sequence $(\alpha_{i})_{i}$ is finitely generated, the set $\mathcal{R}_{1}$
belongs to a finitely generated lattice. In particular, it is well-ordered by the standard order $<$ on $\mathbb R$.
%\edz{added which order on $\mathbb R$ is meant.}
\\

Now, we repeat the same procedure of elimination for germs (\emph{resp}.
for formal series), with right-hand side $\delta_{1}$ (\emph{resp}.
$\widehat{\delta}_{1}$) instead of $1$. We put $\widehat{R}_{1}=\widehat{\Psi}_2^{\infty}+\widehat{R}_{2}$.
To eliminate the first block from $\widehat{\delta}_{1}$, say $x^{\beta}R(\boldsymbol{\ell}^{-1})$,
we take: 
\begin{equation}\label{eq:psi2}
(\widehat{\Psi}_{2})'=\frac{x^{\beta}R(\boldsymbol{\ell}^{-1})}{x^{\alpha_{1}}P_{1}(\boldsymbol{\ell}^{-1})}=x^{\beta-\alpha_{1}}Q_{2}(\boldsymbol{\ell}),
\end{equation}
where $Q_{2}$ is a rational function.

In the same step, we define the germ $\Psi_2\in\mathcal G_{AN}$:
\begin{equation}
\Psi_{2}(x):=\int_{d}^{x}t^{-(\alpha_{1}-\beta)}Q_{2}(\boldsymbol{\ell}(t))\, dt,
\label{eq:psi}
\end{equation}
where $d>0$ if $\alpha_{1}-\beta>1\text{ or }(\alpha_{1}=\beta+1,\ \text{ord}(Q_{2})\leq1)$, and $d=0$ if $\alpha_{1}-\beta>1\text{ or }(\alpha_{1}=\beta+1,\ \text{ord}(Q_{2})\leq1)$. Obviously, $\Psi_{2}^{\infty}\in\mathcal G_{AN}$. Note that
there are no new singularities created in $Q_{2}(\boldsymbol{\ell})$, since
the denominator of $\frac{R(\boldsymbol{\ell}^{-1})}{P_{1}(\boldsymbol{\ell}^{-1})}$
is just a positive integer power of $P_{1}(\boldsymbol{\ell}^{-1})$.

Note here that, depending on the order of the right-hand side in \eqref{eq:psi2} (that is, depending on the step of the algorithm), in \eqref{eq:psi} we get either $\Psi_2(x)\to\infty$ (infinite block) or $\Psi_2(x)=o(1)$ (infinitesimal block), as $x\to 0$. In the former case, we may choose the constant in $\Psi_2$ arbitrarily (any small $d>0$), while in the latter we put $d=0$ in order to have $\Psi_2(0)=0$ and thus to avoid constant terms in the infinitesimal part of the Fatou coordinate. 
 We will show below that there are only finitely many infinite steps, but at most countably many infinitesimal steps. 

Repeating the same procedure, we conclude that the blocks in $\widehat{\delta}_{2}$ are of the form $x^{\beta}R(\boldsymbol{\ell}^{-1})$,
where $R$ is again a rational function whose denominator can only be a \emph{positive}
integer power of the polynomial $P_{1}(\boldsymbol\ell^{-1})$, and 
%\edz{Explained important properties.}
and 
\[
\beta\in\mathcal{R}_{2}:=\{(\alpha_{n_{1}}+\cdots+\alpha_{n_{k}}-k)-2(\alpha_{1}-1),\ k\in\mathbb{N},\ k\geq2\}\cup\mathcal{R}_{1}.
\]
At the $r$-th step of this procedure, the powers of $x$ in $\widehat{\delta}_{r}$
are: 
\[
\beta\in\mathcal{R}_{r}:=\cup_{p\in\mathbb{N},\, p\leq r}\{(\alpha_{n_{1}}+\cdots+\alpha_{n_{k}}-k)-p(\alpha_{1}-1),\ k\in\mathbb{N},\ k\geq p\}.
\]
Therefore, the monomials appearing in the algorithm on the right-hand
side of the modified Abel equation are always of the form $x^{\beta}R(\boldsymbol{\ell}^{-1})$,
where 
\begin{align*}
\beta&\in\mathcal{R}:=  \bigcup_{r\in\mathbb{N}}\{(\alpha_{n_{1}}+\cdots+\alpha_{n_{k}}-k)-r(\alpha_{1}-1)|\ k\in\mathbb{N},\ k\geq r\}=\\
= & \{(\alpha_{n_{1}}-\alpha_{1})+\cdots+(\alpha_{n_{r}}-\alpha_{1})+(\alpha_{n_{r+1}}-1)+\cdots+(\alpha_{n_{k}}-1),\ r\leq k,\ r,\, k\in\mathbb{N}\}.
\end{align*}
That is, the set $\mathcal{R}$ is the set of \emph{all} finite sums of nonegative elements of the form $(\alpha_{i}-\alpha_{1})$ and
$(\alpha_{i}-1)$, where $\alpha_{i}\geq1$, $i\in\mathbb{N}$, is
the sequence of powers of $x$ in the Dulac expansion $\widehat{f}(x)$.
Since $(\alpha_{i})_{i}$ belong to a finitely generated lattice,
it is the same for the elements of $\mathcal{R}$. In particular,
$\mathcal{R}$ is well-ordered. Its order type is $\omega$, and its
elements form a sequence tending to $+\infty$. Since all the powers
of $x$ in the common support of all the right-hand sides $\delta$
in the course of the algorithm belong to $\mathcal{R}$, they can
either be ordered in an \emph{infinite strictly increasing sequence
tending to $+\infty$} or there are \emph{only finitely many} of them.
In the latter case, the block by block algorithm terminates in finitely
many steps. Otherwise, it needs $\omega$ steps to terminate. In any
case, the construction by blocks of the formal Fatou coordinate is
not transfinite, as it terminates in at most $\omega$ steps.

Furthermore, thanks to the direct relation of $\widehat{\Psi}_{r}$
and the leading block of $\widehat{\delta}_{r}$ described in \eqref{eq:psi2},
and by Proposition~\ref{prop:rational_fct}, we also see that the
support of $\widehat{\Psi}$ is well-ordered.

\subsubsection{\bf{The principal (\emph{infinite}) part of the Fatou coordinate}}

\label{sec:subone}\
 Let $\alpha_{1}>1$ be such that $\text{ord}(\mathrm{id}-\widehat{f})=(\alpha_{1},m)$,
$m\in\mathbb{Z}^{-}$, as above. We have proved in Section~\ref{sec:subzero}
that the orders of the blocks on the right-hand sides of the Fatou equation in the
course of eliminations belong to a finitely generated lattice. The
order of the leading block on the right-hand side in every step strictly
increases. Therefore, it follows that after \emph{finitely many steps}
of block by block eliminations, the order of the right-hand side $\widehat{\delta}$
becomes strictly bigger than $\alpha_{1}-1$.

We denote by $r_{0}\in\mathbb{N}$ the smallest number such that,
after $r_{0}$ steps, the Abel equation becomes: 
\[
R_{r_{0}}(f(x))-R_{r_{0}}(x)=o(x^{\alpha_{1}-1}).
\]

The $r_{0}$-th step is the \emph{critical step} between the infinite
and the infinitesimal part of the Fatou coordinate. That is: 
\[
R_{r_{0}-1}(x)\to\infty,\ R_{r_{0}}(x)\to 0 \text{ as }x\to0.
\]
This is a direct consequence of the following Proposition~\ref{prop:aux}:

\begin{prop}[Order of the blocks in the algorithm]\label{prop:aux}
Let $\beta\in\mathcal{R}$ be the order of the leading block of the
right-hand side $\widehat{\delta}_{k}\in\widehat{\mathcal{L}}$ of
the equation 
\[
\widehat{R}(\widehat{f})-\widehat{R}=\widehat{\delta}_{k}.
\]
Then the leading block $\widehat{\Psi}_{k}\in\widehat{\mathcal{L}}_{2}^{\infty}$ of $\widehat R$ is
a block of order $\beta-(\alpha_{1}-1)$. Moreover, 

a) if $\beta>\alpha_{1}-1$, then $\widehat{\Psi}_{k}\in\widehat{\mathcal{L}}$, 

b) if $\beta<\alpha_{1}-1$, then $\widehat{\Psi}_{k}\in\widehat{\mathcal{L}}^{\infty}$, 

c) if $\beta=\alpha_{1}-1$, then $\widehat{\Psi}_{k}\in\widehat{\mathcal{L}}_{2}^{\infty}$. 

\end{prop}

\begin{proof} By Taylor expansion, as described in the algorithm
in Section~\ref{sec:subzero}, 
\[
\widehat{\Psi}_{k}=\int x^{-(\alpha_{1}-\beta)}\widehat{Q}(\boldsymbol{\ell})\, dx,
\]
where $\widehat{Q}$ is an asymptotic expansion of a rational function,
and $\int$ denotes the formal integral. The result now follows by
Proposition~\ref{prop:rational_fct}. \end{proof}

It follows that there exists an index $r_{0}\in\mathbb N$, called the \emph{critical
index}, such that all the infinite blocks of $\Psi$ or $\widehat \Psi$
are exactly those (finitely many) indexed by $r\leq r_{0}$. Hence we define \emph{the
principal part} of $\Psi\in \mathcal G_{AN}$ or $\widehat{\Psi}\in\widehat{\mathcal L}_2^\infty$ by: 
\begin{align*}
 & \Psi^{\infty}=\Psi_{1}+\cdots+\Psi_{r_{0}},\\
 & \widehat{\Psi}^{\infty}=\widehat{\Psi}_{1}+\cdots+\widehat{\Psi}_{r_{0}}.
\end{align*}
Note that, by the integral definition \eqref{eq:psi} (1) of infinite blocks $\Psi_i$, $i=1\ldots r_0$, due to arbitrary choices of $d>0$, $\Psi^\infty$ defined above is unique only up to an additive constant. Therefore, by Proposition~\ref{prop:rational_fct}, parallel construction and by the definition of integral sections and sectional asymptotic expansions in Section~\ref{sec:asymptotic_expansions}, 
the formal principal part $\widehat{\Psi}^{\infty}\in\widehat{\mathcal{L}}_{2}^{\infty}$
obtained by the blockwise integration by parts is the sectional asymptotic
expansion of $\Psi^{\infty}$ with respect to any integral section $\mathbf s$, up to appropriate choices of constant terms in both $\Psi^\infty$ and $\widehat\Psi^\infty$. Also, the change of the integral section $\mathbf s$ leads to the change in constant terms in $\Psi^\infty$ or in $\widehat\Psi^{\infty}$.
\medskip{}

\subsubsection{\bf{The infinitesimal part of the Fatou coordinate}}

\label{sec:subtwo}

Let $r_{0}$ be the \emph{critical index} defined at the end of Section~\ref{sec:subone}.
The germ $R=\Psi-\Psi^{\infty}$, $R\in\mathcal G_{AN}$, satisfies the difference equation:
\begin{equation}
R(f(x))-R(x)=\delta(x),\label{eq:abel}
\end{equation}
where $\delta(x)=O(x^{\gamma})$, 
with $\gamma>\alpha_{1}-1$, $\delta\in\mathcal G_{AN}$. Note that this is the first step of the
block by block algorithm for which we obtain an \emph{infinitesimal}
solution. That is, $R=o(1)$, as $x\to0$.

On the one hand, we continue solving formally block by block (expanding
integrals by integration by parts). We have already proved at the
end of Section~\ref{sec:subzero} that we terminate the formal block
by block algorithm in countably many steps: 
\begin{equation}
\widehat{R}=\sum_{i\in\mathbb{N}}\widehat{R}_{r_{0}+i}.\label{eq:er}
\end{equation}
By Proposition~\ref{prop:aux}, $\widehat{R}_{r_{0}+i}\in\widehat{\mathfrak{L}}$,
$i\in\mathbb{N}$, are blocks of strictly increasing orders, so
$\widehat{R}\in\widehat{\mathcal{L}}$.

On the other hand, in each step we get a germ $R_{r_{0}+i}\in\mathcal G_{AN},\ i\in\mathbb{N},$
which is defined by an appropriate integral in \eqref{eq:psi}, with $d=0$. Note that $R_{r_0+i}=o(1)$, as $x\to 0$, as a consequence of the fact that we have put $d=0$.

Therefore, by Proposition~\ref{prop:rational_fct}, $\widehat{R}_{r_{0}+i}$ 
is exactly the sectional asymptotic expansion of $R_{r_{0}+i}$ in $\widehat{\mathcal{L}}$ with respect to any integral section.

Instead of proving that the infinite series of analytic germs from $\mathcal G_{AN}$ is an
analytic germ from $\mathcal G_{AN}$, once that we have reached the equation for the infinitesimal
part \eqref{eq:abel}, we directly construct a germ $R\in\mathcal{G}_{AN}$, $R(x)=o(1)$, satisfying \eqref{eq:abel},
by following the classical construction explained in \cite{loray}. We
define: 
\begin{equation}
R(x):=-\sum_{k=0}^{\infty}\delta\big(f^{\circ k}(x)\big)\label{eq:sumsol}.
\end{equation}
We prove in Proposition~\ref{prop:asyproof} that the above sum with $\delta\in\mathcal G_{AN}$ and $\delta(x)=O(x^{\gamma})$, 
$\gamma>\alpha_{1}-1$,  converges for sufficiently small $d>0$ to an analytic function defined on $(0,d)$. That is, that $R$ defined by \eqref{eq:sumsol} is well-defined and $R\in\mathcal G_{AN}$. Directly
by Proposition~\ref{prop:asyproof} below, we get that $R(x)=o(1)$, as
$x\to 0$. %Take $\varepsilon>0$ small enough
%such that $\alpha_{1}+\varepsilon-1<\gamma$. By Proposition~\ref{prop:estimate} below,
%we see that $0<f^{\circ k}(x)<\big(\frac{k}{2}\big)^{-\frac{1}{\alpha_{1}+\varepsilon-1}}$,
%$x\in(0,d)$, for $d$ sufficiently small. By \emph{Weierstrass theorem},
%$R\in\mathcal G_{AN}$. It is easy to check that
%$R$ satisfies \eqref{eq:abel}. \\

Now, $\Psi=\Psi^{\infty}+R$, $\Psi\in\mathcal G_{AN}$, is a Fatou
coordinate for $f$.

\begin{prop}\label{prop:asyproof} Let $f\in\mathcal G_{AN}$ be the Dulac germ as 
above. Let $\alpha_{1}>1$ be the order of the leading block in the Dulac expansion of $\,\, \mathrm{id}-f$.
Let $\delta\in\mathcal G_{AN}$ be an analytic germ on some open interval $(0,d),\ d>0$, satisfying $\delta(x)=O(x^{\gamma}),\ \gamma>\alpha_{1}-1$, obtained in the block-by-block construction described above.
Then $h$ defined by the series 
\begin{equation}
h(x):=-\sum_{k=0}^{\infty}\delta\big(f^{\circ k}(x)\big)\label{eq:wei}
\end{equation}
belongs to $\mathcal G_{AN}$. Moreover, for every $\varepsilon>0$,
$h(x)=o(x^{\gamma-(\alpha_{1}-1+\varepsilon)})$, as $x\to0$.\end{prop}
\begin{proof}
The idea is to follow the procedure described in e.g. Loray \cite{loray} for regular germs at $0$, and to apply Weierstrass theorem. To be able to apply Weierstrass theorem, we need to extend $f$ and $\delta$ from $(0,d)\subset \mathbb R^+$ analytically to an open, $f$-invariant set $U\subset\mathbb C$ in the complex plane, chosen such that the sum \begin{equation}\label{eq:cpl}h(z):=-\sum_{k=0}^{\infty}\delta\big(f^{\circ k}(z)\big)\end{equation} converges normally on $U$. Then, in particular, the sum  $-\sum_{k=0}^{\infty}\delta\big(f^{\circ k}(x)\big)$ converges for sufficiently small $x\in(0,d)$ to $h\in\mathcal{G}_{AN}$. 

The \emph{quasi-analyticity} of the Dulac germ $f\in\mathcal G_{AN}$ (see \cite[Definition 2 and Corollary 3]{ilyalim}) means that $f$ can be extended from $(0,d)$ to an  analytic function $f(z)$ defined on a standard quadratic domain $Q$, as defined in Ilyashenko \cite{ilyalim}. 

%%%
Beware that here we work in the original chart  and not in the logarithmic chart as Ilyashenko does.  More precisely,
we restrict to an open sectorial domain $V_P$ at the origin, centered at the positive real line.

%%%

\begin{comment}
Note that here under standard quadratic domain $Q$ we suppose the pre-image by $e^{-z}$ of the Ilyashenko's standard quadratic domain defined in the logarithmic chart. It is spiraling neighborhood of the origin in the Riemann surface of the logarithm, and it is invariant under $f$. Since we are interested only in the neighborhood of $\mathbb R^+$, we restrict only to the principal level corresponding to the principal branch of the logarithm. On the principal level, we take some open sector (of sufficiently small opening less than $2\pi$ and sufficiently small radius) centered at $\mathbb R^+$, and denote it by $V_P$. We have seen that $f(x)$ can from $(0,d)$ be extended analytically to $V_P$. 
\end{comment}
We construct now an open set $U\subset V_P$ containing $(0,d)$, $d>0$, which is \emph{invariant under $f$}. Moreover we prove the existence on $U$ of bounds which guarantee normal convergence of \eqref{eq:cpl}. We adapt the construction given in Loray \cite{loray}. 

Moreover, the function $\delta$ can be analytically extended to $U\subset V_P$. Indeed, the function $\delta$ is the sum of finitely many integrals, see \eqref{eq:diff}. 
Each of these integrals is well-defined and analytic on $V_P$. Moreover, by \cite[Corollary 3]{ilyalim}, the same Dulac asymptotic expansion at $0$ of $f$ holds on the whole domain $V_P$. Therefore, the asymptotic behavior at $0$ of $\delta$ on $V_P$ coincides with its behavior  on the real line.

Let us now return to the construction of the $f$-invariant open set $U\subset V_P$ where we show the normal convergence of the series.

Let $\widehat f(z)=z-az^{\alpha_1}\boldsymbol\ell^m+h.o.t.$, $a\in\mathbb R$, be the Dulac expansion of $f$. We make a change of coordinates 
\begin{equation}\label{eq:Psi_0}
w=\Psi_0(z)=\frac{1}{a(\alpha_1-1)}z^{-\alpha_1+1}\boldsymbol\ell^{-m}.
\end{equation}
Note that $\Psi_0$ is the leading term in the expansion of the Fatou coordinate of $f$. The idea is to transform $f$  to \emph{almost} a translation by $1$. It can be shown that this change of variables is a well-defined, univalued analytic homeomorphism $\Psi_0:V_P\to\Psi_0(V_P)$, 
for an appropriate choice of a small sector 
 $V_P$ centered at $(0,d)$. Indeed, up to a scalar multiplication, we can write $\Psi_0$ as the composition:
\begin{equation}\label{eq:decomp}
\Psi_0=f_3\circ f_2\circ f_1,\ f_1(z)=z^{-\frac{\alpha_1-1}{m}},\ f_2(z)=z\cdot (-\log z),\ f_3(z)=z^{m}.
\end{equation}
All three functions are analytic homeomorphisms on corresponding domains.  To see that $f_2$ is an analytic homeomorphism, we can consider the function $h(z)=-\log f_2(e^{-z})=z-\log z$, which is a small perturbation of identity  at infinity.

We construct now an $f$-invariant domain inside $V_P$. We work in the $w$-plane at $\infty=\Psi_0(0)$ (see \eqref{eq:Psi_0}). We show that there exists an open subset of the image $\Psi_0(V_P)$ which is $\tilde f$-invariant, where $\tilde f$ is the analogue of $f$ in the $w$-plane, defined in \eqref{eq:eff} below. Its pre-image by $\Psi_0$ is then an open subset $U$ of $V_P$ in the $z$-plane, which contains $(0,d)$ and which is $f$-invariant. 

Let us analyze the \emph{shape} of $\Psi_0(V_P)$. % to be able to conclude that it contains the image of an $f$-invariant domain in $V_P$. 
Since $\Psi_0$ is a homeomorphism from $V_P$ to $\Psi_0(V_P)$, $\Psi_0(V_P)$ is open and connected, since $V_P$ is. In fact, by decomposition \eqref{eq:decomp} of $\Psi_0$, we see that $\Psi_0$ maps sector $V_P$ to \emph{almost} a sector, that is, the image $\Psi_0(V_P)$ contains a sector at $\infty$ (of sufficiently big radius and small opening). Denote this sector by $W_{R_1,\theta_1}\subset \Psi_0(V_P)$.

Let $f_0$ be the germ from $\mathcal G_{AN}$ with the Fatou coordinate $\Psi_0$, that is, defined by
$$
f_0(z):=\Psi_0^{-1}\big(\Psi_0(z)+1\big),
\ z\in V_P.
$$
The map $\Psi_0$ is analytic on $V_P$ and $\Psi_0(z)+1$ remains in the image $\Psi_0(V_P)$ for every $z\in V_P$, as we show below when analyzing the image $\Psi_0(V_P)$, after possibly reducing the radius and the opening of $V_P$. Hence, $f_0$ is well-defined and analytic on $V_P$. By the Taylor expansion and explicit formula for $\Psi_0$, we get that $f_0(z)=z-az^{\alpha_1}\boldsymbol\ell^m+o(z^{\alpha_1}\boldsymbol\ell^{m})$. Note that the leading term of  the expansion of $f_0-\mathrm{id}$ is the same as the leading term of $f-\mathrm{id}$.

We define $\tilde f(w):=\Psi_0\circ f\circ \Psi_0^{-1}(w)$ (an analogue of $f$ in the $w$-plane, defined on the image by $\Psi_0$ of an $f$-invariant subset of $V_P$). In this new coordinate $w=\Psi_{0}\left(z\right)$, it can be checked that there exists a germ $h(w)$ such that:
\begin{align}\label{eq:eff}
\tilde f(w)&=\Psi_0(f(z))=\Psi_0(f_0(z)+R(z))\nonumber\\
&=\Psi_0(f_0(z))+\Psi_0'(f_0(z))\cdot R(z)+h.o.t.\nonumber\\
&=\Psi_0(z)+1+g(z)=w+1+h(w),
\end{align}
and $h(w)=o(1)$ as $w\to+\infty$. The function $R$ above is defined by $f(z)=f_0(z)+R(z)$, so $R$ is analytic on $V_P$ with $R(z)=cz^{\alpha_2}\boldsymbol\ell^{k}+o(z^{\alpha_2}\boldsymbol\ell^{k}),\ c\in\mathbb R,\ (\alpha_2,k)\succ (\alpha_1,m),\ z\to 0,\ z\in V_P$. Thus,
\begin{align}\label{eq:eh}h(w)=\begin{cases}C(\log w)^{-p}+o((\log w)^{-p}),\ p\in\mathbb N,&\ 
\text{if $\alpha_2=\alpha_1$},\\
Cw^{-\beta}\log^r  w+o(w^{-\beta}\log^r w),\ \beta>0,\ r\in\mathbb R,&\ \text{if $\alpha_2>\alpha_1,\ \ C\in\mathbb R$}.
\end{cases}\end{align}
It is easier to get estimates  and $\tilde f$-invariance in the $w$-plane. Let $R>0$. If $|w|>R$ then from \eqref{eq:eh} we get:
\begin{equation}\label{eq:ocj}
|h(w)|\leq \begin{cases}c (\log R)^{-p},&\ \text{if }\alpha_2=\alpha_1,\\
cR^{-(\beta-\varepsilon)},\ \varepsilon>0,\ \beta-\varepsilon>0, &\ \text{if }\alpha_2>\alpha_1,\ \ c>0.
\end{cases}
\end{equation}

We now construct an $\tilde{f}$-invariant sector $W_{R,\alpha_R}\subset\Psi_0(V_P)$. To this end we define $\alpha_R>0$ such that 
\begin{equation}\label{eq:alphaR}
\sin\alpha_R:=
\begin{cases} 
c (\log R)^{-p},&\ \text{if }\alpha_2=\alpha_1,\\
cR^{-(\beta-\varepsilon)},\ \varepsilon>0,&\ \text{if }\alpha_2>\alpha_1.
\end{cases}
\end{equation}
Take $R$ big enough so that the open sector $W_{R,\alpha_R}$ in the $w$-plane of radius $|w|>R$ and of opening $2\alpha_R$ is contained in $W_{R_1,\theta_1}$ ($R>R_1$ and $\alpha_R<\theta_1$). This is possible, since $\alpha_R\to 0$, as $R\to\infty$. 

By construction, on $W_{R,\alpha_R}$ it holds that, for every $w\in W_{R,\alpha_R}$, the sector centered at $w$ with opening $2\alpha_R$ is completely contained in $W_{R,\alpha_R}$. Using \eqref{eq:eff}, \eqref{eq:ocj} and \eqref{eq:alphaR}, we obtain, for every $w$ such that $|w|\geq R$, that the whole orbit $\{\tilde f^{\circ n}(w):\ n\in\mathbb N_0\}$ remains in the open sector centered at $w$ with opening $2\alpha_R$. Therefore, the open sector $W_{R,\alpha_R}$ is invariant for $\widetilde f(w)$. 

Moreover, for $w\in W_{R,\alpha_R}$ we get, by \eqref{eq:eff}, \eqref{eq:ocj}, the following estimates:
$$
|\tilde f^{\circ n}(w)-w-n|\leq \begin{cases}c n R^{-(\beta-\varepsilon)},\\
c n (\log R)^{-p}.\end{cases}
$$
It follows that:
\begin{equation}\label{eq:jj}
|\tilde f^{\circ n}(w)|\geq C_R \cdot n,\ C_R>0,\ w\in W_{R,\alpha_R}.
\end{equation}
The sector $W_{R,\alpha_R}$ is open and contains $(R,+\infty)\subset\mathbb R_+$. Denote by $U\subset\mathbb C$ its pre-image in the $z$-plane by the homeomorphism $\Psi_0^{-1}$:
$$
U:=\Psi_0^{-1}(W_{R,\alpha_R}).
$$
Obviously, $U$ is an open set, containing $(0,d)$, and invariant under $f$. From \eqref{eq:jj}, returning to the variable $z$ we get:
$$
|f^{\circ n}(z)|\leq D_R\cdot n^{-(\alpha_1-1-\varepsilon)},\ n\in\mathbb N, \ D_R>0,\ z\in U,
$$
for any small $\varepsilon>0$ such that $\alpha_1-1-\varepsilon>0$. 
Using this estimate, $f$-invariance of $U$ and the fact that $\delta(z)=O(z^\gamma),\ \gamma>\alpha_1-1$, $z\in U$, it holds that the series \eqref{eq:wei} is well-defined and converges normally on $U$. By the Weierstrass theorem, the function $h$ defined by the series \eqref{eq:wei} is analytic on $U$. In particular, $h\big|_{(0,d)}\in\mathcal G_{AN}$. 
\medskip

Having shown the convergence of the series \eqref{eq:wei} towards an analytic germ $h\in\mathcal G_{AN}$, the rest of the proof can be carried through on $(0,d)$. Let $\varepsilon>0$ such that $\gamma>\alpha_{1}-1+\varepsilon$.
By Proposition~\ref{prop:estimate}, we get that there exists $d>0$
such that 
\[
f^{\circ k}(x)\leq\big(\frac{k}{2}\big)^{-\frac{1}{\alpha_{1}-1+\varepsilon}},\ x\in(0,d).
\]
The last point of the proposition
follows from the following inequalities: 
\begin{align*}
|h(x)|=\Big|-\sum_{k=0}^{\infty} & \delta(f^{\circ k}(x))\Big|\leq C\sum_{k\geq0}x^{\gamma}(1+\frac{k}{2}x^{\alpha_{1}-1+\varepsilon})^{-\frac{\gamma}{\alpha_{1}-1+\varepsilon}}\\
 & \leq C_{1}x^{\gamma}\int_{0}^{\infty}(1+\frac{t}{2}x^{\alpha_{1}+\varepsilon-1})^{-\frac{\gamma}{\alpha_{1}-1+\varepsilon}}\, dt\leq O(x^{\gamma-(\alpha_{1}-1+\varepsilon)}),\ x\to0+,
\end{align*}
where $C>0$ and $C_{1}>0$ are constants. \end{proof}

\begin{prop}\label{prop:estimate} Consider an analytic function $f$ on
$(0,d)$ satisfying $0<f(x)<x$, $x\in(0,d)$. Let $\alpha>0$ such that, for every
$\varepsilon>0$:
$$\lim_{x\to0}\frac{x^{\alpha+\varepsilon}}{x-f(x)}=0.$$
Then for every
$\varepsilon>0$ there exist $n_{0}\in\mathbb{N}$ and $d_{1}>0$,
such that 
\begin{equation}
0<f^{\circ n}(x)<x\Big(1+\frac{n}{2}x^{\alpha+\varepsilon-1}\Big)^{-\frac{1}{\alpha+\varepsilon-1}},\ x\in(0,d_{1}),\ n\geq n_{0}.\label{eq:estima}
\end{equation}
\end{prop}

\begin{proof} Take $\varepsilon>0$ and $f_{1}(x)=x-x^{\alpha+\varepsilon}$.
We prove the proposition for $f_{1}$, and the statement for $f$
follows. To prove the estimate, we use the change of variables $w=\frac{1}{(\alpha+\varepsilon-1)x^{\alpha+\varepsilon}}$,
$w\in(M,\infty)$, by which $f_{1}$ becomes: 
\[
F_{1}(w)=w+1+O(w^{-1}).
\]
It is now easy to end the proof working with $F_{1}$. \end{proof}

\subsubsection{\bf{The sectional asymptotic expansions of the Fatou coordinate with respect to integral sections}}

\label{sec:subfour}

We prove here that $R\in \mathcal G_{AN}$ defined by \eqref{eq:sumsol} admits the formal infinitesimal part  of the Fatou coordinate $\widehat R\in\widehat{\mathcal L}$ constructed in \eqref{eq:er} as its sectional asymptotic expansion with respect to \emph{any} integral section $\mathbf s$. We have already proven in Subsection~\ref{sec:subone} that $\Psi^\infty\in\mathcal G_{AN}$ admits $\widehat \Psi^{\infty}\in\widehat{\mathcal L}_2^\infty$ as its sectional asymptotic expansion with respect to any integral section for appropriate choices of constant terms in $\Psi^\infty,\ \widehat\Psi^\infty$.  Consequently, the Fatou coordinate $\Psi=\Psi^\infty+R\in\mathcal G_{AN}$ will admit the formal Fatou coordinate $\widehat\Psi=\widehat\Psi^\infty+\widehat R\in\widehat{\mathcal L}_2^\infty$ as its sectional asymptotic expansion with respect to any integral section, with appropriate choice of constant terms in  $\Psi^\infty,\ \widehat\Psi^\infty$. Finally, different choices of integral sections lead to different choices of the constants in $\Psi$ or in $\widehat \Psi$, if $\widehat\Psi$ is to be the sectional asymptotic expansion of $\Psi$ with respect to the new integral section.
\smallskip

Put $h_{n}:=R-\sum_{i=1}^{n}R_{r_{0}+i}$, $n\in\mathbb{N}$. Obviously,
$h_{n}\in\mathcal G_{AN}$ and $h_n=o(1)$ (since $R=o(1)$ and $R_{r_{0}+i}=o(1)$, $i\in\mathbb N$). It can easily be checked that $h_n$
satisfies the difference equation with the 
right-hand side $\delta_{n}(x)=O(x^{\gamma_{n}})$, where $\gamma_{n}\to\infty$ and $\delta_n\in\mathcal G_{AN}$.
Iterating the equation for $h_{n}$ and passing to limit as for $R(x)$ before, we get that
$h_{n}$ is necessarily given by the formula: 
\[
h_{n}(x)=-\sum_{k=0}^{\infty}\delta_{n}(f^{\circ k}(x)).
\]
By Proposition~\ref{prop:asyproof}, we get that $h_{n}=O(x^{\beta_{n}})$,
where $\beta_{n}\to\infty$. Together with the fact that $\widehat{R}_{r_{0}+i}\in\widehat{\mathfrak{L}}$
is the sectional asymptotic expansion of $R_{r_{0}+i}\in\mathcal G_{AN}$, with respect to any integral section (see Proposition~\ref{prop:rational_fct}), $i\in\mathbb{N}$, 
this proves that $R\in \mathcal G_{AN}$, $R=o(1)$, admits $\widehat{R}\in\widehat{\mathfrak{L}}$ as the sectional asymptotic expansion
with respect to any integral section.

\subsubsection{\bf{Uniqueness of the Fatou coordinate}}

\label{sec:subthree}\

By Proposition~\ref{def:formfatou}, the formal Fatou coordinate $\widehat{\Psi}\in\widehat{\mathcal{L}}_{2}^{\infty}$
is unique in $\widehat{\mathfrak L}$, up to an additive constant.

\begin{example}[Non-uniqueness of a Fatou coordinate with only sectional asymptotic expansion in $\mathfrak L$]\label{ex:add}
%\edz{Added the example. CHECK!}
Let $\Psi\in\mathcal G_{AN}$ be the Fatou coordinate for $f$ Dulac constructed in the algorithm. It admits a sectional asymptotic expansion $\widehat\Psi\in\widehat{\mathcal L}_2^\infty$ with respect to an integral section:
$$
\widehat \Psi=\sum_{i=1}^{\infty} x^{\alpha_i}\widehat f_i(\boldsymbol\ell),
$$
where $\alpha_1<0$. The integral section attributes to partial expansion 
%\edz{changed}
 $\sum_{i=1}^{n}x^{\alpha_i}\widehat f_i(\boldsymbol\ell)$ at each limit ordinal stage $n\in\mathbb N$ the \emph{sum} $\sum_{i=1}^{n}x^{\alpha_i} f_i(\boldsymbol\ell)\in\mathcal G_{AN}$, where $f_i\in\mathcal G_{AN}$ are integral sums of $\widehat f_i\in\widehat{\mathcal L}_0^I$.

Take now another Fatou coordinate for $f$:
$$
\Psi_1=\Psi+\sin(2\pi\Psi).
$$
Obviously, $\Psi_1\in\mathcal G_{AN}$ and satisfies the Abel equation. We prove that $\Psi_1$ admits the same $\widehat\Psi$ as its sectional asymptotic expansion if we take e.g. the section $\mathbf s$ that maps:
\begin{align*}
& x^{\alpha_1}\widehat f_1\mapsto x^{\alpha_1}\big(f_1(\boldsymbol\ell)+e^{\frac{\alpha_1}{\boldsymbol\ell}}\sin (2\pi\Psi(e^{-1/\boldsymbol\ell}))\big),\\
&x^{\alpha_i}\widehat f_i\mapsto x^{\alpha_i}f_i,\ i\geq 2.
\end{align*}
Note that the real sine function is bounded, so $e^{\frac{\alpha_1}{\boldsymbol\ell}}\sin (2\pi\Psi(e^{-1/\boldsymbol\ell}))$, $\alpha_1<0$, is exponentially small with respect to $\boldsymbol\ell$. Therefore $f_1(\boldsymbol\ell)+e^{\frac{\alpha_1}{\boldsymbol\ell}}\sin (2\pi\Psi(e^{-1/\boldsymbol\ell}))$ admits Poincar\' e asymptotic expansion $\widehat f_1(\boldsymbol\ell)$. To conclude, $\Psi_1(x)$ admits the sectional asymptotic expansion $\widehat\Psi$ with respect to section $\mathbf s$ (which is not integral sectional).
\end{example}

%\edz{This proof is changed completely . Read. Check for mistakes. }
We prove now that the Fatou coordinate for a Dulac germ $f$ admitting a sectional asymptotic expansion in $\widehat{\mathfrak L}$ with respect to an integral section \emph{is} unique (up to an additive constant). Suppose the contrary, that is, that $\Psi_1\in\mathcal G_{AN}$ and $\Psi_2\in\mathcal G_{AN}$ are two Fatou coordinates for $f$ that admit integral sectional asymptotic expansions in $\widehat{\mathfrak L}$. Note that if $\Psi_1$ and $\Psi_2$ admit integral sectional asymptotic expansions in $\widehat{\mathfrak L}$, then $\Psi_1-\Psi_2$ admits an integral sectional asymptotic expansion in $\widehat{\mathfrak L}$ (this can be easily seen by definition of integral sectional asymptotic expansions as sum of blocks $\sum_{i\in\mathbb N}\int_*^x x^{\gamma_i}R_i(\boldsymbol\ell) dx$, where $R_i$ is a rational function, and $\gamma_i\in\mathbb R$ strictly increase as $i\to \infty$). This is not the case for only sectional asymptotic expansions in $\widehat{\mathfrak L}$. Take, for example, $\Psi$ and $\Psi_1$ from Example~\ref{ex:add} which both admit sectional asymptotic expansions in $\widehat{\mathfrak L}$, but their difference, $\sin(2\pi\Psi)$, obviously does not admit an asymptotic expansion in $\widehat{\mathfrak L}$.

The difference $\Psi_1-\Psi_2\in\mathcal G_{AN}$ satisfies the equation:
$$
(\Psi_1-\Psi_2)(f(x))=(\Psi_1-\Psi_2)(x).
$$
Since $\Psi_1-\Psi_2$ admits an integral sectional asymptotic expansion, we have that (by blocks) there exists $\beta\in\mathbb R$, a rational function $R$ not identically equal to $0$ and $\delta>0$ such that
$$
\Psi_1(x)-\Psi_2(x)=\int_*^x R(\boldsymbol\ell){x^{\beta}}+o(x^{\beta+1+\delta}),\ x\to 0.
$$
Since $f(x)=x+x^{\alpha_1}P_1(\boldsymbol\ell)+o(x^{\alpha_1+\delta})$, $x\to 0$ , by Taylor formula we get:
$$
(\Psi_1-\Psi_2)'(\xi_x)\cdot \big(x^{\alpha_1}P_1(\boldsymbol\ell)+o(x^{\alpha_1+\delta})\big)=0,
$$
where $\xi_x=x+o(x)$, $x\to 0$. We conclude $(\Psi_1-\Psi_2)'(\xi_x)=0$ for small $x$, but this is not possible since $(\Psi_1-\Psi_2)'(\xi_x)\sim x^{\beta} R(\boldsymbol\ell)\sim x^{\beta}\boldsymbol\ell^M,$ $M\in\mathbb Z$, as $x\to 0$. The last approximation follows since $R(\boldsymbol\ell)$ is a rational function not identically equal to $0$.

%%%%%%
\subsection{\bf{A precise form of the formal Fatou coordinate}}

In the course of the proof of the Theorem (in Subsection~\ref{sec:subzero}), we have also proved a more  precise form of the formal Fatou coordinate $\widehat\Psi$ as stated below in Proposition~\ref{prop:exformfatou}. In particular, we have proved that there is only one monomial in $\widehat\Psi$ which contains the double logarithm.
\begin{prop}[The formal Fatou coordinate of a parabolic Dulac germ]\label{prop:exformfatou} Let
$\widehat{\Psi}\in\widehat{\mathcal{L}}_{2}^{\infty}$ be the $($unique in $\widehat{\mathfrak L}$ up to an additive constant$)$ formal Fatou coordinate for a parabolic Dulac
germ $f\in\mathcal G_{AN}$. Then there exists $\rho\in\mathbb R$ such that $\widehat{\Psi}-\rho\boldsymbol{\ell}_{2}^{-1}\in\widehat{\mathfrak{L}}^{\infty}$,
and 
\begin{equation}\label{eq:ffd}
\widehat{\Psi}-\rho\boldsymbol{\ell}_{2}^{-1}=\sum_{i=1}^{\infty}x^{\alpha_{i}}\widehat{f}_{i}(\boldsymbol{\ell}).
\end{equation}
Here,

1. $\alpha_{1}<0$,

2. $\alpha_{i}$ is a strictly increasing real sequence tending to $+\infty$ (finitely generated),

3. $\widehat{f}_{i}\in\widehat{\mathcal L}_0^\infty$ is a formal Laurent series given by:
\begin{equation}\label{eq:conve}
\widehat{f}_{i}(\boldsymbol{\ell})=\frac{\int x^{\alpha_{i}-1}\widehat R_{i}(\boldsymbol{\ell})\,dx}{x^{\alpha_{i}}},
\end{equation}
where $\widehat R_{i}$ is a \emph{convergent series} $($more precisely, the power asymptotic
expansion of a rational function$)$.
\end{prop}

Note that, in general, $\widehat{f}_{i}(y)$ is a \emph{divergent
power series} for every positive value $y>0$. Nevertheless, it is what we call
\emph{integrally summable} in Definition~\ref{defi}. 

\medskip     

\section{Appendix}

\label{sec:appendix}

The following Remark is used in Section~\ref{sec:embedding}.

\begin{obs}[Description of the formal Fatou coordinate for parabolic transseries]\label{rem:class}
We explain here another way of deducing that the formal Fatou coordinate $\widehat\Psi\in\widehat{\mathfrak L}$ of a parabolic $\widehat{f}\in\widehat{\mathcal{L}}$
belongs to the class $\widehat{\mathcal{L}}_{2}^{\infty}$ and its precise form.

Recall the formal normal form $\widehat{f}_{0}$ of $\widehat{f}$
in $\widehat{\mathcal{L}}$ deduced in \cite{mrrz2}, given as the
formal time-one map of a simple vector field: 
\begin{align*}
 & \widehat{f}=x+ax^{\alpha}\boldsymbol{\ell}^{m}+\mathrm{h.o.t.},\text{ with }\alpha>1,\text{ or }\alpha=1\text{ and }m\in\mathbb{N},\\
 & \widehat{f}_{0}(x)=\mathrm{Exp}\big(\widehat \xi_{0}(x)\frac{\mathrm{d}}{\mathrm{d}x}\big)\cdot\mathrm{id},\\
&\qquad\qquad\quad \widehat \xi_0(x)=\frac{ax^{\alpha}\boldsymbol{\ell}^{m}}{1+\frac{a\alpha}{2}x^{\alpha-1}\boldsymbol{\ell}^{m}-\big(\frac{am}{2}+\frac{b}{a}\big)x^{\alpha-1}\boldsymbol{\ell}^{m+1}}\frac{\mathrm{d}}{\mathrm{d}x},\ b\in\mathbb{R}.
\end{align*}
By the existence part of Proposition~\ref{def:formfatou}, we look
for a formal Fatou coordinate of $\widehat{f}_{0}$ as the formal antiderivative
of $\frac{1}{\widehat \xi_{0}}$: 
\begin{align*}
\widehat{\Psi}_{0}= & \frac{1}{a}\int x^{-\alpha}\boldsymbol{\ell}^{-m}\, dt+\frac{\alpha}{2}\log x+(\frac{m}{2}+\frac{b}{a^{2}})\int\frac{dx}{x\log x}\\
= &\frac{1}{a} \widehat{h}(x)-\frac{\alpha}{2}\boldsymbol{\ell}^{-1}+(\frac{m}{2}+\frac{b}{a^{2}})\boldsymbol{\ell}_{2}^{-1}+C,\ C\in\mathbb{R}.
\end{align*}
Here, $\widehat{h}\in\widehat{\mathcal{L}}^{\infty}$ is obtained by repeated formal \emph{integration
by parts}: 
\begin{equation}
\int x^{-\alpha}\boldsymbol{\ell}^{-m}\, dx=\begin{cases}
\frac{1}{1-\alpha}x^{-\alpha+1}\boldsymbol{\ell}^{-m}+\frac{m}{1-\alpha}\int x^{-\alpha}\boldsymbol{\ell}^{-m+1}\, dx, & \alpha\neq1,\\
-\frac{1}{m+1}\boldsymbol{\ell}^{-m-1}+C,\ C\in\mathbb{R}, & \alpha=1,\ m\in\mathbb{N}.
\end{cases}\label{intparts}
\end{equation}
In particular, in the case when $m\in\mathbb N$, $\widehat h$ contains only finitely many monomials. 
\medskip

Now put $\widehat{\Psi}:=\widehat{\Psi}_{0}\circ\widehat{\varphi}$,
where $\widehat{\varphi}\in\widehat{\mathcal{L}}$ parabolic is the
formal change of variables reducing $\widehat{f}$ to $\widehat{f}_{0}$. It is easy to check that the composition $\widehat\Psi$ belongs to $\widehat{\mathcal{L}}_{2}^{\infty}$. It is obviously a Fatou coordinate for $\widehat f$, since $\widehat\Psi_0$ satisfies the Abel equation for $\widehat f_0$. Moreover, by Proposition~\ref{def:formfatou}, the formal Fatou coordinate of $\widehat f$ is unique in $\widehat{\mathfrak{L}}$ (up to a constant term).

Therefore, at most one term with double logarithm $\boldsymbol{\ell}_{2}^{-1}$
appears in the formal Fatou coordinate $\widehat{\Psi}$ of $\widehat f$. It corresponds to the residual term $bx^{2\alpha-1}\boldsymbol\ell^{2m+1}$
in the normal form $\widehat{f}_{0}(x)=x+ax^\alpha\boldsymbol\ell^m+bx^{2\alpha-1}\boldsymbol\ell^{2m+1}$. 
\end{obs} 
\smallskip

\subsection{Propositions for Section~\ref{sec:asymptotic_expansions}}\

\emph{Proof of Proposition~\ref{alph}.}
Suppose that $\widehat f$ is divergent and that there are two exponents of integration $\alpha\neq\beta$ for $\widehat f$. Then 
$$\frac{d}{dx}\big(x^{\alpha}\widehat f(\boldsymbol\ell)\big)=x^{\alpha-1}  R_1(\boldsymbol\ell),\ \frac{d}{dx}\big(x^{\beta}\widehat f(\boldsymbol\ell)\big)=x^{\beta-1}  R_2(\boldsymbol\ell).$$
Therefore,
\begin{align*}
\frac{d}{dx}\big(x^{\alpha}&\widehat f(\boldsymbol\ell)\big)=\frac{d}{dx}\big(x^{\alpha-\beta}\cdot x^{\beta}\widehat f(\boldsymbol\ell)\big)=(\alpha-\beta)x^{\alpha-1} \widehat f(\boldsymbol\ell)+x^{\alpha-\beta}\frac{d}{dx}\big(x^{\beta}\widehat f(\boldsymbol\ell)\big)=\\
&=(\alpha-\beta)x^{\alpha-1} \widehat f(\boldsymbol\ell)+x^{\alpha-1} R_2(\boldsymbol\ell)\\
&\qquad \Rightarrow x^{\alpha-1} R_1(\boldsymbol\ell)=(\alpha-\beta)x^{\alpha-1} \widehat f(\boldsymbol\ell)+x^{\alpha -1} R_2(\boldsymbol\ell).
\end{align*}
Since $R_1$ and $R_2$ are both convergent Laurent series and $\alpha-\beta\neq 0$, this is a contradiction with divergence of $\widehat f$.
\hfill $\Box$

%%%%
\begin{prop}\label{pomoc} Let $\alpha\in\mathbb{R}$, $m\in\mathbb{Z}$.
Let 
\[
a(x):=
\int_{d}^{x}t^{-\alpha}\boldsymbol{\ell}^{m}\, dt, 
\]
with  $d>0$ if $\alpha>1\text{ or }(\alpha=1,\ m\leq1)$,
and $d=0$ if  $\alpha<1\text{ or }(\alpha=1,\ m>1)$. The integral is divergent at $0$ in the first case and convergent in the second.
Then $a\in\mathcal G_{AN}$ and
\[
a(x)=\begin{cases}
O(x^{-\alpha+1}\boldsymbol{\ell}^{m}),\ x\to0+, & \alpha\neq1,\\
\frac{\boldsymbol{\ell}^{m-1}}{m-1}+C,\ C\in\mathbb{R}, & \alpha=1,\ m<1,\\
\boldsymbol{\ell}_{2}^{-1}+C,\ C\in\mathbb{R}, & \alpha=1,\ m=1,\\
\frac{\boldsymbol{\ell}^{m-1}}{1-m}, & \alpha=1,\ m>1.
\end{cases}
\]
\end{prop}

The proof is based on elementary calculus.\\

The following proposition is an easy consequence of Proposition~\ref{pomoc}
and integration by parts: 

\begin{prop}\label{prop:rational_fct} \

\noindent
$(1)$ Let  $\widehat{R}\in\widehat{\mathcal{L}}_{0}^\infty$ and let $n_{0}:=\rm{ord}(\widehat{R})$, $n_{0}\in\mathbb{Z}$.
For $\alpha\in\mathbb{R}$ consider the transseries  $\widehat{F}$ and $\widehat{f}$ respectively defined by:
\begin{equation}\label{eq:as1}\widehat{F}\left(x\right)=\frac{\int x^{\alpha-1}\widehat{R}\left(\boldsymbol{\ell}(x)\right)dx}{x^{\alpha}}\text{ and }\widehat{f}\left(y\right)=\widehat{F}\left(e^{-1/y}\right),
\end{equation} where the numerator of $\widehat{F}\left(x\right)$ is the formal antiderivative in $\widehat{\mathfrak L}$ of $x^{\alpha-1}\widehat{R}\left(\boldsymbol{\ell}(x)\right)$ without constant term (as explained in the footnote on p. \pageref{ha}).
%\edz{Added} 
Then 
$\widehat{F}\in\widehat{\mathcal{L}}_{2}^{\infty}$ and $\widehat{f}\in\widehat{\mathcal{L}}_{1}^{\infty}$.
Moreover: 
\begin{enumerate}
\item[a)] if $\alpha\ne0$ then $\widehat{f}\in\widehat{\mathcal{L}}_{0}^{\infty}$ and
\[
\widehat{f}(y)=\sum_{n=0}^{\infty}a_{n}y^{n_{0}+n},\ a_{n}\in\mathbb{R};
\]
\item[b)]if $\alpha=0$ then
\[
\widehat{f}(y)=\begin{cases}
\sum_{n=0}^{-n_{0}}a_{n}y^{n_{0}-1+n}+b\log y+\sum_{n=0}^{\infty}b_{n}y^{n}, & n_{0}\leq1,\\
\sum_{n=0}^{\infty}a_{n}y^{n_{0}-1+n}, & n_{0}>1,\ a_n,b_n,b\in\mathbb R.
\end{cases}
\]

\end{enumerate}
\medskip

\noindent
$(2)$ Let $R\in \mathcal G_{AN}$
admit the integer power asymptotic expansion $\widehat{R}\in\widehat{\mathcal{L}}_{0}^\infty$. Let $f\in\mathcal G_{AN}$ be given by: 
\begin{equation}\label{eq:as}
f(y):=
\frac{\int_{d}^{e^{-1/y}}s^{\alpha-1}R\big(\boldsymbol\ell(s)\big)\, ds}{e^{-\frac{\alpha}{y}}},
\end{equation}
where $d>0$ if  $\alpha<0$ or $(\alpha=0$ and  $n_{0}\leq1)$, and $d=0$ if  $\alpha>0$ or $(\alpha=0$ and  $n_{0}>1)$.

Then $\widehat f\in\widehat{\mathcal L}_0^\infty$ defined in \eqref{eq:as1} is the power asymptotic expansion of $f\in\mathcal G_{AN}$ if $\alpha\neq 0$. If $\alpha=0$, $\widehat f$ is the Poincar{\'e} asymptotic expansion of $f$ up to a constant.
\end{prop} 
%%%%%%%%%%

\medskip
\subsection{Proofs of propositions from Section~\ref{sec:embedding}}

\begin{proof}[Proof of Proposition~\ref{def:formfatou}]
\noindent \emph{}

\emph{The chain rule.} As a prerequisite for the proof, we prove that
the chain rule is valid in our formal setting. That is, if $\widehat{\Psi}\in\widehat{\mathcal{L}}_{2}^{\infty}$ and $\{\widehat f^t\},\ \widehat f^t\in\widehat{\mathcal L}$ is a $C^1$-flow as defined in \cite[Def.1.2]{mrrz2}, then:
\[
\frac{\mathrm{d}}{\mathrm{d}t}\big(\widehat{\Psi}(\widehat{f^{t}}(x))\big)=\widehat{\Psi}'\big(\widehat{f}^{t}(x)\big)\cdot\frac{\mathrm{d}}{\mathrm{d}t}\widehat{f}^{t}(x)
\]
holds formally in $\widehat{\mathcal{L}}_{2}^{\infty}$. Here, $\frac{\mathrm{d}}{\mathrm{d}t}$
applied to a transseries means the derivation monomial by monomial. Since $\widehat{f}^{t}(x)=x+\mathrm{h.o.t.}$ with coefficients
in $C^{1}(\mathbb{R})$, it stems from \emph{Neuman's Lemma} (see
\cite{dries}) that the coefficients of $\widehat{\Psi}(\widehat{f^{t}}\left(x\right))$
also belong to $C^{1}\left(\mathbb{R}\right)$.

It is sufficient to prove the equality in $\widehat{\mathcal{L}}_{2}^{\infty}$ for a single monomial $m(x)$
from the support $\mathrm{Supp}(\widehat{\Psi})$: 
\[
\frac{{\displaystyle \mathrm{d}}}{\mathrm{d}t}(m(\widehat{f}^{t}(x)))=m'(\widehat{f}^{t}(x))\cdot\frac{\mathrm{d}}{\mathrm{d}t}\widehat{f}^{t}(x).
\]
Both sides
share a common well-ordered support. Take any monomial from this support. By \emph{Neumann's
lemma}, on both sides only finitely many monomials from $\widehat{f}^{t}$
contribute to it. Now the equality holds if we replace $\widehat{f}^{t}$
by the finite sum of its terms corresponding to these first monomials.
Therefore, the coefficients of every monomial on both sides coincide.
\medskip

\emph{The existence.} Take $\widehat{\Psi}$ to be the formal antiderivative in $\widehat{\mathfrak L}$ without constant term (as explained in the footnote on p. \pageref{ha})
%\edz{Added!}
of $1/\widehat{\xi}$, where 
\[
\widehat{\xi}:=\frac{\mathrm{d}}{\mathrm{d}t}\widehat{f}^{t}\Big|_{t=0}.
\]
We prove that $\widehat{\Psi}$ satisfies the equation \eqref{eq:rfatou}
for the formal Fatou coordinate. Integrating formally $1/\widehat{\xi}$ (every monomial is formally integrated by parts),
we conclude that $\widehat{\Psi}\in\widehat{\mathcal{L}}_{2}^{\infty}$.
Indeed, since $\widehat{\xi}\in\widehat{\mathcal{L}}$, we get that
$1/\widehat{\xi}\in\widehat{\mathcal{L}}^{\infty}$. In the integration
process, the double logarithm $\boldsymbol\ell_2^{-1}$ is generated when integrating the monomial $x^{-1}\boldsymbol\ell$.

Using $\frac{{\displaystyle \mathrm{d}}}{\displaystyle\mathrm{d}t}\widehat{f}^{t}=\widehat{\xi}(\widehat{f}^{t})$,
by the chain rule proved above, we get that 
\[
\frac{d}{dt}\big(\widehat{\Psi}(\widehat{f}_{t}(x))\big)=\widehat{\Psi}'(\widehat{f}^{t}(x))\cdot\frac{\mathrm{d}}{\mathrm{d}t}\widehat{f}^{t}(x)=\frac{1}{\widehat{\xi}(\widehat{f}^{t}(x))}\cdot\frac{\mathrm{d}}{\mathrm{d}t}\widehat{f}^{t}(x)=1.
\]
Integrating this equality from $0$ to $t$ gives the equality \eqref{fatou}.
In particular, $\widehat{\Psi}(\widehat{f}(x))-\widehat{\Psi}(x)=1$.
\medskip

\emph{The uniqueness.} Suppose that there exist two formal Fatou coordinates
$\widehat{\Psi}_{1},\ \widehat{\Psi}_{2}\in\widehat{\mathfrak{L}}$.
Let $\widehat{\Psi}:=\widehat{\Psi}_{1}-\widehat{\Psi}_{2}$. Then
$\widehat{\Psi}\in\widehat{\mathfrak{L}}$, that is, $\widehat{\Psi}\in\widehat{\mathcal{L}}_j^\infty$ for some $j\in\mathbb N_0,$ and it satisfies:
\[
\widehat{\Psi}(\widehat{f}(x))-\widehat{\Psi}(x)=0.
\]
Since $\widehat{f}\in\widehat{\mathcal{L}}$, by Taylor expansion
in $\widehat{\mathcal{L}}_{j}^{\infty}$, we get: 
\[
\widehat{\Psi}'\cdot\widehat{g}+\frac{1}{2!}\widehat{\Psi}''\cdot\widehat{g}^{2}+\cdots=0.
\]
If $\widehat{\Psi}'\neq0$ in $\widehat{\mathcal{L}}_{j}^{\infty}$,
since $\text{ord}(\widehat{g})\succ(1,0)$, the leading term of the
left-hand side is the non-zero leading term of $\widehat{\Psi}'\cdot\widehat{g}$,
which is a contradiction. Therefore, $\widehat{\Psi}'=0$, so $\widehat{\Psi}=C,\ C\in\mathbb{R}$ (which can be easily checked by analysing leading terms or looked up in \cite{dries}
%\edz{I added. Check if this really exists in Dries. Should we explain in more detail?}
).
\end{proof}

\begin{proof}[Proof of Proposition \ref{prop:exi}]
%\edz{I changed the proof. Compare and check.}
\noindent \emph{} We prove both directions. 

$1.$ Suppose that an analytic Fatou coordinate $\Psi$ for $f$ exists
on $(0,d)$. By definition it is strictly monotonic on $(0,d)$. Since $f(x)<x$, $\Psi$ is strictly decreasing on $(0,d)$. Therefore its image is the interval $(\Psi(d),+\infty)$. Then the
family $\{f^{t}\}$ of analytic functions on $(0,d)$, $d>0$, defined
by: 
\begin{equation}
f^{t}(x):=\Psi^{-1}\big(\Psi(x)+t\big),\label{eq:fat1}
\end{equation}
is well-defined and analytic for $t\in (\Psi(d)-\Psi(x), +\infty)$, $x\in (0,d)$. Note that, for every $x\in(0,d)$, due to monotonicity of $\Psi$, $\Psi(d)-\Psi(x)<0$, and $t$ can be extended to $+\infty$ in the positive direction. We conclude that $f^1$ exists for all $x\in(0,d)$ and $\{f^t\}$ gives
a $C^{1}$-flow in which $f$ embeds as the time-one map.

Furthermore, 
\[
\xi:=\frac{\mathrm{d}}{\mathrm{d}t}f^{t}\Big|_{t=0}=\frac{\mathrm{d}}{\mathrm{d}t}\Psi^{-1}\big(\Psi(x)+t\big)\Big|_{t=0}=\frac{1}{\Psi'}.
\]
Since $\Psi$ is a strictly monotonic on $(0,d)$, either $\Psi'>0$ or $\Psi'<0$
in $(0,d)$, so $\xi$ is non-oscillatory in $(0,d)$.
\medskip

$2.$ The vector field whose flow is given by the $C^{1}$-family $\{f^{t}\}$, $t\in\mathbb{R}$, of analytic functions on $(0,d)$, is
given by the formula: 
\[
X=\xi(x)\frac{\mathrm{d}}{\mathrm{d}x},\text{ where }\xi:=\frac{\mathrm{d}}{\mathrm{d}t}f^{t}\Big|_{t=0}.
\]
Obviously, $\xi$ is also analytic on $(0,d)$. Take $\Psi$ to be
the antiderivative of $1/\xi$. That is, $\Psi'=\frac{1}{\xi}.$ We
prove that $\Psi$ is a Fatou coordinate for $f$, that is, satisfies
\eqref{eq:fat0}. We solve the differential equation for the flow:
\begin{align*}
\dot{x} & =\xi(x)\\
\frac{\mathrm{d}x}{\xi(x)} & =\mathrm{d}t,\\
t & =\int_{x}^{f^{t}(x)}\frac{ds}{\xi(s)}.
\end{align*}
We get that 
%\edz{Changed to $(\mathbb R,0)$.}
\[
\Psi(f^{t}(x))-\Psi(x)=t,\ x\in (0,d),\ t\in(\mathbb{R}^+,0) \text{ (the exact interval depending on $x$)}.
\]
In particular, $\Psi(f(x))-\Psi(x)=1$ (note that by assumption $f$ embeds in $\{f^t\}$ on $(0,d)$ as time-$1$ map, so $f$ is defined for every $x\in(0,d)$.

Moreover, since $\Psi'=\frac{1}{\xi}$, and $\xi$ does not change
sign in some interval $(0,d)$, $\Psi$ is strictly monotonic in the
same interval.
\end{proof}
\begin{obs}[The importance of \emph{non-oscillatority} in Proposition~\ref{prop:exi}]\label{rem:oscil}Consider
the flow $\left\{ f^{t}\right\} _{t}$ of an analytic vector field
$X=\xi\frac{\mathrm d}{\mathrm{dx}}$ on $\left(0,d\right)$. Take a
non-singular point $x_{0}>0$ of the vector field ($\xi(x_0)\neq 0$). Then, by \eqref{eq:fat2}, the Fatou coordinate
$\Psi_{x_{0}}$ is defined at a point $x$ as the time $t\in\mathbb{R}$
such that $f^{t}\left(x_{0}\right)=x$. In particular, $\Psi_{x_{0}}\left(x_{0}\right)=0$.
Obviously, $\Psi_{x_{0}}$ cannot be defined at any singular (equilibrium) point
of vector field $\xi$. 

For example, the flow $\left\{ f^{t}\right\} _{t\in\mathbb{R}}$ of
the analytic vector field $X=x^{2}\sin(1/x)\frac{\mathrm{d}}{\mathrm{d}x}$
on $(0,d)$ consists of analytic maps on $(0,d)$. But, as
the vector field $X$ admits infinitely many singular points which
accumulate at the origin, we cannot define a Fatou coordinate on any
interval $\left(0,d_{1}\right)$. \end{obs} 
\medskip

\textbf{Acknowledgements.} \emph{The authors would like to thank the referee for very useful suggestions and comments.}
\emph{Address:}$\quad$ $^{1}$ and $^{3}$ : Universit\'e de Bourgogne,
D\'epartment de Math\'ematiques, Institut de Math\'ematiques de
Bourgogne, B.P. 47 870, 21078 Dijon Cedex, France 

$^{2}$ : University of Zagreb, Faculty of Science, Department of Mathematics, Bijeni\v cka 30, 10000 Zagreb, Croatia

$^{4}$ : University of Zagreb, Faculty of Electrical Engineering and Computing, Department of Applied
Mathematics, Unska
3, 10000 Zagreb, Croatia
\end{document}